\documentclass{amsart}

\usepackage{amsmath,amssymb,amsthm}
\usepackage{enumitem}
\usepackage{fancyhdr}
\usepackage[colorlinks=true]{hyperref}
\usepackage[all]{xy}

\def\public{2}

\if\public0
\usepackage[notcite,notref]{showkeys}
\fi

\usepackage{words}
\usepackage{sets}

\newtheorem*{theorem*}{Theorem}
\newtheorem{proposition}{Proposition}
\newtheorem{corollary}[proposition]{Corollary}
\newtheorem{lemma}[proposition]{Lemma}
\theoremstyle{remark}
\newtheorem{remark}[proposition]{Remark}
\newtheorem{example}[proposition]{Example}
\theoremstyle{definition}
\newtheorem{definition}[proposition]{Definition}

\numberwithin{proposition}{section}

\if\public0
\newcommand{\marg}[1]{\normalsize{{\footnote{{#1}}}}{\marginpar[\hfill\tiny\thefootnote$\rightarrow$]{{$\leftarrow$\tiny\thefootnote}}}}
\else
\newcommand{\marg}[1]{}
\fi

\if\public1
\fi

\begin{document}

\title{Obstruction theories and virutal fundamental classes}
\author{Jonathan Wise}
\date{\today}
\maketitle

\if\public1
\thispagestyle{fancy}
\fi

\begin{abstract}
We give a new definition of an obstruction theory for infinitesimal deformation theory and relate it to earlier definitions of Artin, Fantechi--Manetti, Li--Tian, and Behrend--Fantechi.
\end{abstract}

\tableofcontents

\section{Introduction} \label{sec:intro}

An infinitesimal deformation problem asks, given some structure $\xi$ over a scheme $S$ and an infinitesimal extension $S'$ of $S$, whether it is possible to extend $\xi$ to a structure of the same type over $S'$.  In any reasonable example, $\xi$ may be viewed as a section of a stack $X$, and the deformation problem can be viewed as one of extending a map:
\begin{equation} \label{eqn:20} \xymatrix{
S \ar[r]^\xi \ar[d] & X \\
S' \ar@{-->}[ur]_{\xi'} .
} \end{equation}
Such a problem may be ``solved'' by describing an obstruction group, $\Obs(S,J)$, depending in a functorial manner on $S$ and $J$, but not on the specific extension $S'$, and an obstruction class $\omega \in \Obs(S,J)$ that does depend on all of the data and is also functorial in $S$ and $J$.  Should $\omega$ vanish, the dashed arrows completing Diagram~\eqref{eqn:20} form a torsor under a second group, $\Def(S, J)$, again depending only on $S$ and $J$.  

While there is a canonical description of the group $\Def(S,J)$, there may be many different different choices for $\Obs(S,J)$.  It was shown by Fantechi and Manetti~\cite{FM} that there is at least a \emph{minimal} example of an obstruction theory, but the minimal example is rarely the most natural or convenient to use; moreover, it is necessary to make use of other obstruction theories to define the virtual fundamental class~\cite{LT,BF} for use in virtual enumerative geometry.

Despite the utility of obstruction theories in the construction of the virtual fundamental class and in Artin's criteria for the algebraicity of stacks, it remains unclear just what an obstruction theory \emph{is}.  There are at least three different definitions in use, all of which are likely inequivalent.  The purpose of this paper is to propose a definition of an obstruction theory that, at least in the author's opinion, captures all of the properties one would want from an obstruction theory without sacrificing any of the generality of the definitions already available.

We will also compare this definition with the others.  The first of these is Artin's, which was give to streamline his criteria for algebraicity of stacks \cite[Section~(2.6)]{versal}.  Fantechi and Manetti generalized Artin's definition to a relative setting~\cite{FM}.\footnote{In fact, Fantechi and Manetti gave a definition that applies even to functions whose ``tangent spaces'' are not even required to be vector spaces!  The definition given here can likely be generalized that far as well, but I have not pursued it since I do not know of an application.}

Obstruction theories usually arise cohomologically, and from this point of view the definition of Artin--Fantechi--Manetti definition is phrased in terms of the cohomology groups of an inexplicit cohomology theory.  The finer structure of a cohomology theory itself proved essential in the definition of the virtual fundamental class.

To capture this finer structure, two roughly dual definitions of a perfect obstruction theory are available.  They were introduced at about the same time by Li--Tian \cite{LT} and Behrend--Fantechi \cite{BF}.  The Li--Tian definition is given in deformation-theoretic terms very close in spirit to Artin's original definition; the essential refinement \cite[Definition~1.3]{LT} is the introduction of a perfect complex whose cohomology groups are the tangent and obstruction spaces.

Behrend and Fantechi also work crucially with perfect complexes but they work dually, relying elegantly on the existence of a cotangent complex $\bL_X$.  From their point of view, an obstruction theory is simply a morphism of complexes $\bE_\bullet \rightarrow \bL_{X}$ whose cone is concentrated in degrees $\leq -2$ \cite[Definition~4.4]{BF}.  The deformation-theoretic content of an obstruction theory is neatly encapsulated in its relation to the cotangent complex, via the cotangent complex's celebrated role in deformation theory \cite{Ill1,Ols-cc,Ols-def}.

An obstruction theory in the sense of Behrend--Fantechi induces one in the sense of Li--Tian and these were shown to give equivalent virtual fundamental classes~\cite{KKP}.  A Li--Tian obstruction theory induces an obstruction theory in the sense of Artin--Fantechi--Manetti:
\begin{equation*}
\begin{pmatrix} \text{Behrend} \\ \text{Fantechi} \end{pmatrix} \Rightarrow \begin{pmatrix} \text{Li} \\ \text{Tian} \end{pmatrix} \Rightarrow \begin{pmatrix} \text{Artin} \\ \text{Fantechi} \\ \text{Manetti} \end{pmatrix} .
\end{equation*}
The definition presented here lies between those of Behrend--Fantechi and Li--Tian and our main result shows that, with a sufficient number of adjectives, it is equivalent to the former:
\begin{theorem*} 
An obstruction theory, in the sense presented here, is induced from a perfect obstruction theory in the sense of Behrend--Fantechi if and only if it is \emph{locally of finite presentation} and \emph{quasi-perfect} (see Definition~\ref{def:lfp}).
\end{theorem*}
\noindent The theorem is proved as Proposition~\ref{prop:perf} in Section~\ref{sec:perf}.  

While we believe our definition is inequivalent to Li's and Tian's, it is nonetheless close to it in spirit, and the theorem therefore gives a means by which to show that Li--Tian obstruction theories are induced from Behrend--Fantechi obstruction theories.  This application was the original motivation for this paper.

The reason for the disconnect between the definitions of Behrend--Fantechi and Li--Tian comes down to the cotangent complex.  It is a powerful tool for the construction and manipulation of obstruction theories (see, for example, \cite{KKP,Cos,Man}), but it can also be a liability when it must be used in situations not foreseen by Illusie in~\cite{Ill1,Ill2}.

One such situation arose in the theories of relative and degenerate stable maps.  These theories give a good definition of stable maps into mildly singular schemes by restricting the class of deformations with the ``pre-deformability condition'' \cite{Li1}.  This condition is not open in the space of all deformations, and it therefore obscures the relationship between pre-deformable deformations and the cotangent complex of their moduli space.  Because the definition of Li and Tian is more closely related to deformation theory, it becomes easier (but by no means easy!) to define an obstruction theory in the sense of Li--Tian.  This was accomplished by Li in \cite{Li2}.  In order to compare Li's obstruction theory to others defined more easily using the Behrend--Fantechi formalism in \cite{AMW,CMW} (and to a lesser extent, \cite{ACW}), it was necessary to find a definition that incorporates the strengths of both approaches.

Another situation where reliance on the cotangent complex proved a hindrance is the theory of logarithmic stable maps, which was held back for several years by the lack of the logarithmic cotangent complex needed to imitate the Behrend--Fantechi construction of the obstruction theory for stable maps.  The situation was eventually remedied when Olsson defined the logarithmic cotangent complex, and there are now two different moduli spaces of logarithmic stable maps \cite{Kim,Chen,AC,GS}.  However, the construction of the obstruction theory for stable maps in Section~\ref{sec:stable-maps} makes no reference to the cotangent complex, and therefore may be adapted virtually unchanged to the logarithmic setting as well.



The definition given here is based on that of a \emph{left exact, additively cofibered category}, as given by Grothendieck~\cite{ccct} during his early investigation of the cotangent complex.  Additively cofibered categories are an efficient way of describing \emph{additive} $2$-functors from an additive category into the $2$-category of categories.  Infinitesimal deformations naturally form left exact additively cofibered categories, making this a convenient setting in which to treat infinitesimal deformation problems.

Modulo the many functoriality properties that an obstruction theory must satisfy, our definition is the following.  Consider the relative version of Diagram~\eqref{eqn:20}:
\begin{equation} \label{eqn:25} \xymatrix{
S \ar[r] \ar[d] & X \ar[d]^h \\
S' \ar@{-->}[ur] \ar[r] & Y .
} \end{equation}
Suppose that $S = \Spec A$ and $S' = \Spec A'$ are affine and let $J$ be the kernel of the homomorphism $A' \rightarrow A$.  We write $\Exal_X(A, J)$ for the category of commutative diagrams~\eqref{eqn:25}, including the dashed arrow, in which the ideal of $S$ in $S'$ is $J$.  Let $h^\ast \Exal_Y(A,J)$ be the category of such diagrams, excluding the dashed arrow.  An obstruction theory is an additively cofibered category $\sE(A, -)$ on the category of $A$-modules such that the diagrams
\begin{equation} \label{eqn:19} \xymatrix{
\Exal_X(A,J) \ar[r] \ar[d] & e(A,J) \ar[d] \\
h^\ast \Exal_Y(A,J) \ar[r] & \sE(A,J)
} \end{equation}
are cartesian, with $e(A,J)$ denoting the zero object of $\sE(A,J)$.

We can understand the diagram as saying that for any $\xi \in h^\ast \Exal_Y(A,J)$---that is, any commutative diagram of solid lines~\eqref{eqn:25}---there is an associated obstruction $\omega \in \sE(A,J)$ obtained as the image of $\xi$ under the lower horizontal arrow of~\eqref{eqn:19}.  Furthermore, the isomorphisms between $\omega$ and $e(A,J)$ are in bijection, in a specified way, with the lifts of $\xi$ to $\Exal_X(A,J)$---the dashed arrows completing Diagram~\eqref{eqn:18}.  In particular, $\omega$ is \emph{isomorphic} to zero if and only if a solution to the deformation problem exists.

\subsection{Summary of the paper}

The context for the definitions presented here is that of \emph{homogeneous stacks}.  This is generalization of the homogeneous groupoids of Rim \cite{Rim}.  We define homogeneous stacks in Section~\ref{sec:schl} and explain some of their basic properties.

The definition of an obstruction theory is given in Section~\ref{sec:def}, after some background material on additively cofibered categories.  The main results of this paper are contained in Section~\ref{sec:rep}, where we relate our definition to the cotangent complex in~\ref{sec:cc} and to the definition of Behrend--Fantechi in Corollaries~\ref{cor:BF1} and~\ref{cor:BF2}.  Proposition~\ref{prop:perf} gives criterion for an obstruction theory to be representable by a complex of perfect amplitude in degress $[-1,0]$, which is useful in \cite{CMW}.

In Section~\ref{sec:obgrp} we explain how to recover an obstruction sheaf from an obstruction theory as defined here.  This permits us to relate our definition to Artin's and to Li's and Tian's.  In Section~\ref{sec:compat}, we define compatible obstruction theories and relate this notion to the one defined by Manolache~\cite{Man}.  Finally, in Section~\ref{sec:ex},  we give some familiar examples of obstruction theories, viewed from the perspective presented here.  

Perhaps the most interesting example of an obstruction theory is the one for degenerate stable maps that motivated this paper.  It is too long to give a thorough treatment here, and will be described in a separate paper \cite{predef}.  Many details of a related obstruction theory are given in \cite{CMW} as well.  

The appendix contains some background material on bilinear maps of vector bundle stacks.

\subsection{Notational conventions}

As in \cite{versal}, we have tried here to work systematically with covariant functors and cofibered categories over the category of rings instead of contravariant functors and fibered categories over the category of schemes.  In company with \cite{versal} we have not been entirely successful in this respect.

The reason we have used the algebraic language instead of the geometric one is our systematic reliance on the category $\ComRngMod$ of pairs $(A,I)$ where $A$ is a commutative ring and $I$ is an $A$-module.  The opposite of this category, which might most naturally be described as the category of paris $(S, F)$ where $S$ is an affine scheme and $F$ is an object of the \emph{opposite category} of quasi-coherent sheaves on $S$, is a syntactic atrocity.  Furthermore, we will find various categories that have are contravariant with respect to $\ComRng$, but whose natural extensions to schemes are only covariant with respect to \emph{affine} morphisms.  Focusing only on affine schemes allows us to avoid this technicality and shorten a number of statements.

All rings and algebras in this paper are commutative and unital, but we often emphasize the commutativity anyway.  If $A$ is a commutative ring and $J$ is an $A$-module then we write $\tJ$ for the associated sheaf on the small \'etale site of $A$, which is denoted $\et(A)$.  Likewise, $\tA$ denotes the structure sheaf on $\et(A)$.  We employ the same notation for complexes of $A$-modules.

In this paper, the distinction between chain complexes and cochain complexes is purely a matter of notation:  every complex $\bE$ is simultaneously a chain complex $\bE_\bullet$ and a cochain complex $\bE^\bullet$ with $\bE_i = \bE^{-i}$.  When referring to the degrees of a complex, chain or cochain, we always use cohomological degrees.  Therefore to say that a chain complex $\bE_\bullet$ is concentrated in degrees $[-1,0]$ means that $\bE_i = 0$ for $i \not= 0,1$, or, equivalently, that $\bE^i = 0$ for $i \not= -1, 0$.

Finally, note that we often use the following convention:  if $\sE$ is cofibered category over $X\hAlgMod$ (defined in Section~\ref{sec:mod}) then we write $\usE(A,I)$ for the cofibered category over the category $A\hEtAlg$ of \'etale $A$-algebras.  We use the same notation for the corresponding fibered category on the small \'etale site of $\Spec A$.  We also extend this notation to apply to sheaves.

\section{Homogeneous stacks} \label{sec:schl}

We introduce the stacks for which we will describe obstruction theories later in the paper.  Although we are ultimately interested in obstruction theories for moduli spaces that are Deligne--Mumford stacks, these obstruction theories can often be constructed by including these stacks in much larger stacks that cannot be algebraic; moreover, one application of obstruction theories is in demonstrating that stacks are algebraic, so we must have a definition that applies to stacks that are not known a priori to be algebraic.  The stacks we work with will be called \emph{homogeneous}, expanding the terminology employed by \cite{Rim} to a more general setting.

Let $\ComRng$ be the category of commutative rings.  A cofibered category over $\ComRng$ is called a \emph{stack} if it satisfies descent in the \'etale topology.  Note that our stacks are cofibered over commutative rings instead of fibered over schemes, as is more conventional, and the fibers of our stacks are \emph{not} required to be groupoids.

We will say that a cofibered category $F \rightarrow \ComRng$ is \emph{homogeneous} or \emph{satisfies Schlessinger's conditions} if, whenever
\begin{equation*} \xymatrix{
  A' \ar[r] \ar[d] & B' \ar[d] \\
  A \ar[r] & B
} \end{equation*}
is a cartesian diagram in $\ComRng$ and the kernel of $B' \rightarrow B$ is nilpotent, the natural map
\begin{equation*}
  F(A') \rightarrow F(A) \fp_{F(B)} F(B')
\end{equation*}
(which is defined up to unique isomorphism) is an equivalence of categories.  Following \cite[Definition~2.5]{Rim}, this can be put somewhat more intrinsically:
\renewcommand{\labelenumi}{(\roman{enumi})}
\begin{enumerate}
\item if $\eta' \rightarrow \eta$ and $\xi \rightarrow \eta$ are morphism in $F$ such that $p(\eta') \rightarrow p(\eta)$ is an infinitesimal extension, then a fiber product $\xi \fp_\eta \eta'$ exists, and
\item the projection $p : F \rightarrow \ComRng$ preserves such fiber products.
\end{enumerate}

Rim's definition of homogeneity \cite{Rim} efficiently implies all of Schlessinger's criteria \cite[Theorem~2.11]{Schless} simultaneously.  The study of homogeneous groupoids has been taken up again more recently by Osserman (see \cite[Definition~2.7]{Oss}, where ``deformation stacks'' are analogous to homogeneous stacks).

If $X$ is a cofibered category over $\ComRng$, we also write $X\hAlg$ for the underlying category of $X$.  (If $X = \Spec A$ is an affine scheme then $X\hAlg$ is the category of $A$-algebras.)  Accordingly, we refer to the objects of $X\hAlg$ as $X$-algebras.  If $B$ is a commutative ring then we also refer to the objects of $X(B)$ as $X$-algebra structures on $B$.

The following proposition is well-known.
\begin{proposition}
  Any algebraic stack is homogeneous.
\end{proposition}
\begin{proof}
Let $X$ be an algebraic stack, suppose that $A' \rightarrow A$ is an infinitesimal extension of $X$-algebras and $B \rightarrow A$ is a homomorphism of $A$-algebras.  We wish to construct a fiber product of $X$-algebras $B' = B \fp_A A'$.

Let $\oA$, $\oA'$, and $\oB$ denote the underlying commutative rings of $A$, $A'$, and $B$, respectively.  Let $\oB'$ be the fiber product of $\oA'$ and $\oB$ over $\oA$.  The problem is to define a canonical $X$-algebra structure on $\oB'$.  

Since $X$ is a stack in the \'etale topology, this is an \'etale-local problem in $\oB'$, and since the \'etale site of $\oB'$ is the same as the \'etale site of $\oB$, it is also an \'etale local problem in $\oB$.  We can therefore assume that the $X$-algebra structure on $\oB$ is induced from a $C$-algebra structure on $\oB$ for some \emph{smooth} $X$-algebra $C$.  By the infinitesimal criterion for smoothness, we can extend the induced $C$-algebra structure on $\oA$ to a $C$-algebra structure on $\oA'$.  From this lift we obtain a \emph{canonical} $C$-algebra structure on $\oB'$ because $C$ is representable by a commutative ring.

By composition, this induces an $X$-algebra structure on $\oB'$.  We must argue that this construction does not depend on the choice of the smooth $X$-algebra $C$.  Indeed, suppose we had two $X$-algebra structures $B'$ and $B''$ on $\oB'$ that induce the given $X$-algebra structures on $\oA$, $\oA'$, and $\oB$.  Then we obtain an $X \times X$-algebra structure on $\oB'$ such that the induced $X \times X$-algebra structures over $\oA$, $\oA'$, and $\oB$ are compatible induced from the diagonal map $X \rightarrow X \times X$.  But $X$ is an algebraic stack so the diagonal is representable by algebraic spaces.  Therefore the proposition is reduced to the corresponding statement for algebraic spaces.

Repeating the argument again with $X$ being an algebraic space (and in particular with the diagonal of $X$ being an embedding), we discover that we can assume $X$ is representable by embeddings of algebraic spaces.  If we repeat the argument once more, we can assume that $X$ is representable by isomorphisms (the diagonal of an embedding being an isomorphism), in which case there is nothing to prove.
\end{proof}

\subsection{Modules over commutative rings} \label{sec:mod}

Let $\ComRngMod$ be the category of pairs $(A, I)$ where $A$ is a commutative ring and $I$ is an $A$-module.  A map $(A, I) \rightarrow (B, J)$ in $\ComRngMod$ is a ring homomorphism $\varphi : A \rightarrow B$ and an $A$-module homomorphism $I \rightarrow J_{[\varphi]}$.  This category is bifibered (i.e., fibered and cofibered) over $\ComRng$ via the projection sending $(A,I)$ to $A$.  The base change functor, with respect to a morphism $\varphi : B \rightarrow A$ sends the pair $(A,I)$ to $I_{[\varphi]}$ (which is in fact a splitting) and the co-base change functor sends $(B, J)$ to $(A, J \tensor_B A)$.

The fibers of $\ComRngMod$ over $\ComRng$ are abelian categories, and $\ComRngMod$ forms a stack over $\ComRng$ in the fpqc topology.

If $A$ is an $X$-algebra then by definition an $A$-module is a module under the  underlying commutative ring of $A$.  We write $X\hAlgMod$ for the category of pairs $(A, I)$ where $A \in X\hAlg$ and $I$ is a $A$-module.  The morphisms in $X\hAlgMod$ are the evident ones:  a map $(A, I) \rightarrow (B,J)$ is a morphism of $X$-algebras $\varphi : A \rightarrow B$ and a morphism of $A$-modules $I \rightarrow J_{[\varphi]}$.  In more sententious terms, $X\hAlgMod$ is the strict fiber product $X\hAlg \fp_{\ComRng} \ComRngMod$.

\subsection{Tangent vectors} \label{sec:tang}

Suppose that $A$ is a commutative ring and $I$ is an $A$-module.  We may form a new commutative ring $A + I$ whose underlying abelian group is $A \times I$ and multiplication is defined by
\begin{equation*}
  (a, x) . (a', x') = (a a', a x' + a' x).
\end{equation*}
So, in particular, $I \subset A + I$ is an ideal that squares to zero.  This is the \emph{trivial square-zero extension} of $A$ by $I$.

Recall that if $\varphi : B \rightarrow A$ is a homomorphism of commutative rings, then the derivations of $B$ into $I_{[\varphi]}$ can be identified with the lifts of the diagram
\begin{equation*} \xymatrix{
& A + I \ar[d] \\
B \ar@{-->}[ur] \ar[r]^{\varphi} & A .
} \end{equation*}
In other words, the $B$-algebra structures on $A + I$ lifting the specified $B$-algebra structure on $A$ can be identified with the $I$-valued tangent vectors of $\Spec B$ at the $A$-point $\Spec A \rightarrow \Spec B$.  This is a commutative group, which we denote $\sT_B(A,I)$.  The purpose of this section is to extend this definition to the case where $B$ is a homogeneous stack $X$, and make explicit the various functoriality properties of $\sT_X$ that will imply the categories $\sT_X(A, I)$ are abelian $2$-groups.

The construction $(A, I) \mapsto A + I$ gives a new functor from $\Phi : \ComRngMod \rightarrow \ComRng$ (recall that we also have the projection which forgets $I$).  This second functor is faithful  since a morphism $(A, I) \rightarrow (B, J)$ can be recovered from the induced map $A + I \rightarrow B + J$, but it is not full.

Let $X$ be a cofibered category over $\ComRng$.  We may define a cofibered category $\sT_X$ to be the strict fiber product $X\hAlg \fp_{\ComRng} (\ComRngMod, \Phi)$.  The objects of $\sT_X$ are triples $(A, I, \varphi)$ where $(A, I) \in X\hAlgMod$ and $\varphi$ is an $X$-algebra structure on $A + I$.  A morphism in $\sT_X$ from $(A, I, \varphi)$ to $(B, J, \psi)$ is a morphism $u : (A, I) \rightarrow (B, J)$ and a morphism of $X$-algebras $(A + I, \varphi) \rightarrow (B + J, \varphi)$ lifting the map $A + I \rightarrow B + J$ induced from $u$.

\begin{proposition} \label{prop:tang-cof}
  Let $X$ be a cofibered category over $\ComRng$.
  \begin{enumerate}[label=(\roman{*}), ref=(\roman{*})]
  \item $\sT_X$ is cofibered over $X\hAlgMod$ via the projection sending $(A, I, \varphi)$ to $(A, I)$. \label{der:1}
  \item $\sT_X$ is cofibered over $X\hAlg$ via the projection sending $(A, I, \varphi)$ to $A$. \label{der:2}
  \end{enumerate}
  Assume that $X$ is also homogeneous.
  \begin{enumerate}[label=(\roman{*}), ref=(\roman{*}), resume]
  \item $\sT_X$ is fibered over $X\hAlg$ via the projection sending $(A, I, \varphi)$ to $A$. \label{der:3}
  \item For each $A \in X\hAlg$, the category $\sT_X(A)$ is left exact and additively cofibered over $A\hMod$. \label{der:4}
  \item For any $(A, I) \in X\hAlgMod$, the category $\sT_X(A, I)$ is naturally equipped with the structure of an abelian $2$-group. \label{der:5}
  \end{enumerate}
\end{proposition}
\begin{proof}
  \ref{der:1} The functor $\sT_X(A, I) \rightarrow \sT_X(B, J)$ associated to a morphism $u : (A, I) \rightarrow (B, J)$ in $X\hAlgMod$ sends $(A, I, \varphi)$ to $(B, J, v_\ast \varphi)$ where $v : A + I \rightarrow B + J$ is the morphism induced from $u$.

  \ref{der:2} $X\hAlgMod$ is cofibered over $X\hAlg$ and $\sT_X$ is cofibered over $X\hAlgMod$ by Part~\ref{der:1}.

  \ref{der:3} Let $u : A \rightarrow B$ be a homomorphism of commutative rings, and let $J$ be a $B$-module.  Suppose that $\xi \rightarrow \eta$\marg{maybe notation for $\xi$ and $\eta$ should be changed} is a morphism of $X$-algebras lifting $u$.  By homogeneity, the map
  \begin{equation*}
    X(A + J_{[u]}) \rightarrow X(A) \fp_{X(B)} X(B + J)
  \end{equation*}
  is an equivalence.  Restricting, via the projection to $X(A)$, to the fiber over $\xi$, we get an equivalence
  \begin{equation*}
    \sT_X(\xi, J_{[u]}) \rightarrow \sT_X(\eta, J) .
  \end{equation*}
  An inverse to this functor gives the base change functor $\sT_X(\eta, J) \rightarrow \sT_X(\xi, J_{[u]})$.

  \ref{der:4} Let $A$ be a commutative ring.  Suppose that
    \begin{equation*}
      0 \rightarrow I \rightarrow J \rightarrow K \rightarrow 0
    \end{equation*}
    is an exact sequence of $A$-modules.  Then we have a cartesian diagram of rings
    \begin{equation*} \xymatrix{
      A + I \ar[r] \ar[d] & A + J \ar[d] \\
      A \ar[r] & A + K .
    } \end{equation*}
    We therefore deduce that the diagram
    \begin{equation*} \xymatrix{
      \sT_X(A, I) \ar[r] \ar[d] & \sT_X(A, J) \ar[d] \\
      X(A) \ar[r] & \sT_X(A, K)
    } \end{equation*}
    is cartesian by homogeneity.  Restricting to the fiber over $\xi \in X(A)$, we get a cartesian diagram
    \begin{equation*} \xymatrix{
      \sT_X(\xi, I) \ar[r] \ar[d] & \sT_X(\xi, J) \ar[d] \\
      \set{ \xi } \ar[r] & \sT_X(\xi, K)
    } \end{equation*}
    which means that the sequence
    \begin{equation*}
      0 \rightarrow \sT_X(\xi, I) \rightarrow \sT_X(\xi, J) \rightarrow \sT_X(\xi, K) 
    \end{equation*}
is exact.
 
  \ref{der:5} An immediate consequence of \ref{der:4} (in view of \cite[Section~1.4]{ccct} or \cite{gst}).
\end{proof}

\subsection{Square-zero extensions} \label{sec:sq-zero}

Recall that a square-zero extension of a commutative algebra $A$ by an $A$-module $I$ is a surjection of commutative algebras $A' \rightarrow A$ whose kernel consists of elements that square to zero, and an isomorphism between the kernel and $I$, with its $A'$-module structure induced by the projection.  Let $\Exal$ be the category of triples $(A, I, A')$ where $A$ is a commutative ring, and $A'$ is a square-zero extension of $A$ by $I$.  The projection $\Exal \rightarrow \ComRng$ makes $\Exal$ into a fibered category over the category of commutative rings.  There is also a projection $\Exal \rightarrow \ComRngMod$ sending $(A, I, A')$ to $(A, I)$.

If $X$ is a cofibered category over $\ComRng$, define $\Exal_X$ to be the category of triples $(A, I, A')$ where $A$ is an $X$-algebra and $A'$ is a square-zero $X$-algebra extension of $A$ by $I$.  A morphism from $(A, I, A')$ to $(B, J, B')$ is a pair of compatible morphisms of $X$-algebras $\varphi : A \rightarrow B$ and $A' \rightarrow B'$ (which induce a morphism of $A$-modules $I \rightarrow J_{[\varphi]}$.  The following is a more precise form of this definition:  let $\Phi : \Exal \rightarrow \ComRng$ be the projection sending $(A, A')$ to $A'$.  If $X$ is a cofibered category over $\ComRng$, then $\Exal_X = X \fp_{\ComRng} (\Exal, \Phi)$. 

\begin{proposition} \label{prop:def-cof}
  Let $X$ be a cofibered category over $\ComRng$.
  \begin{enumerate}[label=(\roman{*}), ref=(\roman{*})]
  \item $\Exal_X$ is cofibered over $\Exal$. \label{exal:5}
  \item $\Exal_X$ is cofibered over \'etale morphisms in $X\hAlgMod$.  \label{exal:1}
  \end{enumerate}
  Assume that $X$ is also homogeneous.
  \begin{enumerate}[label=(\roman{*}), ref=(\roman{*}), resume]
  \item $\Exal_X$ is fibered over $\ComRng$ and the map $\Exal_X \rightarrow \ComRngMod$ is cartesian. \label{exal:2}
  \item For each $A \in \ComRng$, the category $\Exal_X(A)$ is left exact and additively cofibered over $A\hMod$. \label{exal:3}
  \item For any $(A, I) \in \ComRngMod$, the category $\Exal_X(A, I)$ is naturally equipped with the structure of an abelian $2$-group. \label{exal:4}
  \end{enumerate}
\end{proposition}

\begin{remark}
Although $\Exal_X$ and $\sT_X$ have some similar functoriality properties, the forms of Propositions~\ref{prop:tang-cof} and~\ref{prop:def-cof} are not identical.  This is because $\sT_X(A,I)$ is naturally related to the deformations of $X$, while $\Exal_X(A,I)$ is naturally related to the deformations of $A$.  This can be seen using the cotangent complex, provided one is available:  we can identify the isomorphism classes in $\sT_X(A,I)$ with $\Ext^0(\bL_X, I)$ and we can identify the isomorphism classes in $\Exal_X(A,I)$ with $\Ext^1(\bL_{A/X}, I)$.
\end{remark}

\begin{proof}[Proof of Proposition~\ref{prop:def-cof}]
  \ref{exal:5} This follows immediately from the fact that $X\hAlg$ is cofibered over $\ComRng$.

  \ref{exal:1} We have a commutative diagram
\begin{equation*} \xymatrix{
\Exal_X \ar[r] \ar[d] & \Exal \ar[d] \\
X\hAlgMod \ar[r] & \ComRngMod .
} \end{equation*}
To show that $\Exal_X$ is cofibered over \'etale maps in $X\hAlgMod$, it suffices to show that $\Exal_X$ is cofibered over \'etale maps in $\ComRngMod$, since $X\hAlgMod$ is cofibered over $\ComRngMod$.  Furthermore, $\Exal_X$ is cofibered over $\Exal$ by~\ref{exal:5}, so it is sufficient to show that $\Exal$ is cofibered over \'etale maps in $\ComRngMod$.

One can factor a morphism in $\ComRngMod$ in a unique way into a morphism that is cocartesian over $\ComRng$ followed by a morphism in $A\hMod$ for some commutative ring $A$.  Pushout of exact sequences immediately demonstrates that $\Exal(A)$ is cofibered over $A\hMod$, so the problem reduces to showing that $\Exal$ is cocartesian over \'etale morphisms of $\ComRng$.

  Let $A \rightarrow B \rightarrow C$ be a sequence of morphisms in $\ComRng$ with $A \rightarrow B$ \'etale, and let $\xi \rightarrow \omega$ be a morphism of $\Exal$ over $A \rightarrow C$.  We must show there is a morphism $\xi \rightarrow \eta$ in $\Exal(A \rightarrow B)$, depending only on $\xi$ and the map $A \rightarrow B$, through which the map $\xi \rightarrow \omega$ factors uniquely.  We can represent $\xi$ by a square-zero extension $A'$ of $A$, and $\omega$ by a square-zero extension $C'$ of $C$.  By \cite[Th\'eor\`eme~1.1]{sga4-VIII}, the tensor product induces an equivalence from \'etale $A'$-algebras to \'etale $A$-algebras.  Therefore there is an \'etale $A'$-algebra $B'$, determined uniquely up to unique isomorphism, and a cocartesian diagram
\begin{equation*} \xymatrix{
  A' \ar[r] \ar[d] & B' \ar[d] \\
  A \ar[r] & B.
} \end{equation*}
To check the universality of $B'$, we must show that the commutative diagram of solid lines
\begin{equation*} \xymatrix{
A' \ar@/^10pt/[drr] \ar[d] \ar[dr] \\
A \ar@/^10pt/[drr] \ar[dr] & B' \ar@{-->}[r] \ar[d] & C' \ar[d] \\
& B \ar[r] & C
} \end{equation*}
can be completed in a unique way by a dashed arrow.  This is equivalent to showing that the diagram
  \begin{equation*} \xymatrix{
    C & \ar[l] \ar@{-->}[dl] B' \\
    C' \ar[u] & A' \ar[l] \ar[u] 
  } \end{equation*}
  can be completed in a unique way by a dashed arrow.  But the map $A' \rightarrow B'$ is \'etale and $C' \rightarrow C$ is an infinitesimal extension, so the dashed arrow and its uniqueness are guaranteed by the infinitesimal criterion for being \'etale \cite[D\'efinition~(19.10.2)]{ega4-1}.

  \ref{exal:2} Note first that $\Exal$ is fibered over $\ComRng$ by pullback of exact sequences.  This proves the claim when $X$ is a point.  To prove the general case, we must show that, given any sequence of morphisms $\xi \rightarrow \eta \rightarrow \omega$ in $X$ and square-zero extensions $\xi' \rightarrow \xi$ and $\omega' \rightarrow \omega$, there is a square-zero extension $\eta' \rightarrow \eta$, independent of $\xi$, and a unique factorization of $\xi' \rightarrow \omega'$ through $\eta'$:
  \begin{equation*} \xymatrix{
    \xi' \ar@/^10pt/[drr] \ar@{-->}[dr] \ar[d] \\
    \xi \ar@/^10pt/[drr] \ar[dr] & \eta' \ar[d] \ar[r] & \omega' \ar[d] \\
    & \eta \ar[r] & \omega
  } \end{equation*}
  The diagram illustrates that we are searching for a fiber product $\eta' = \eta \fp_{\omega} \omega'$, which is provided by the homogeneity of $X$.
 
  \ref{exal:3} First we show that if $\xi$ is an $X$-algebra with underlying commutative ring $A$ then $\Exal_X(\xi)$ is additively cofibered over $A\hMod$.  We must show that $\Exal_X(\xi, 0) = 0$, which amounts to noting that $\Exal_X(A, 0) \rightarrow X(A)$ is an equivalence.  We must also check that 
\begin{equation} \label{eqn:3}
  \Exal_X(\xi, I \oplus J) \rightarrow \Exal_X(\xi, I) \times \Exal_X(\xi, J)
\end{equation}
is an equivalence.  We note now that if $\xi' \in \Exal_X(\xi, I \oplus J)$ and $\eta' \in \Exal_X(A, I)$ and $\omega' \in \Exal_X(A, J)$ are the objects induced from $\xi'$, then there is a cartesian diagram
  \begin{equation*} \xymatrix{
    \xi' \ar[r] \ar[d] & \eta' \ar[d] \\
    \omega' \ar[r] & \xi
  } \end{equation*}
  in $X\hAlg$.  Therefore $\xi'$ can be recovered up to unique isomorphism as $\eta' \fp_{\xi} \omega'$, which implies that the morphism displayed in Equation~\eqref{eqn:3} is an equivalence.

  To demonstrate left exactness, suppose that
  \begin{equation*}
    0 \rightarrow I \rightarrow J \rightarrow K \rightarrow 0
  \end{equation*}
  is an exact sequence of $\xi$-modules for an $X$-algebra $\xi$ with underlying commutative ring $A$.  We must show that the sequence
  \begin{equation*}
    0 \rightarrow \Exal_X(\xi, I) \rightarrow \Exal_X(\xi, J) \rightarrow \Exal_X(\xi, K)
  \end{equation*}
  is exact.  Recall that the exactness of this diagram is equivalent to the following diagram's being cartesian:
  \begin{equation} \label{eqn:2} \xymatrix{
    \Exal_X(A, I) \ar[r] \ar[d] & \Exal_X(A, J) \ar[d] \\
    X(A) \ar[r] & \Exal_X(A, K)
  } \end{equation}
  where the maps $\Exal_X(A, I) \rightarrow X(A) \rightarrow \Exal_X(A, K)$ come from the identification $X(A) \simeq \Exal_X(A, 0)$.  An object of the fiber product can be identified with a diagram of solid lines
  \begin{equation} \label{eqn:1} \xymatrix{
    \xi' \ar@{-->}[r] \ar@{-->}[d] & \eta' \ar[d] \\
    \xi \ar[r] & \omega'
  } \end{equation}
  inside $X\hAlg/\xi$ where $\eta' \in X(A, J)$, $\omega' \in X(A, K)$, and $\eta' \rightarrow \omega'$ is cocartesian over $A\hMod$.  By homogeneity, this diagram has a fiber product $\xi'$ completing the diagram above.  Let 
  \begin{equation*} \xymatrix{
    A'_1 \ar[r] \ar[d] & A'_2 \ar[d] \\
    A \ar[r] & A'_3
  } \end{equation*}
  be the image of the cartesian diagram~\eqref{eqn:1} in $\ComRng$.  By homogeneity, this diagram is cartesian.  Furthermore, $A'_2$ is a square-zero extension of $A$ by $J$ and $A'_3$ is a square-zero extension of $A$ by $K$.  It follows that $A'_1$ is a square-zero extension of $A$ by $I$.  This demonstrates an inverse to the natural map $\Exal_X(A, I) \rightarrow \Exal_X(A, J) \fp_{\Exal_X(A, J)} X(A)$ and therefore shows that Diagram~\eqref{eqn:2} is cartesian.

  \ref{exal:4} An immediate consequence of~\ref{exal:3} and \cite{ccct} or \cite{gst}.
\end{proof}

\section{Obstruction theories} \label{sec:def}

\subsection{Additively cofibered categories} \label{sec:acc}

Let $X$ be a stack.  A \emph{stack of cofibered categories}\marg{not sure if this is the right name} on $X$ is a cofibered category $\sE$ over $X\hAlgMod$ that is a stack over $X\hAlg$ in the \'etale topology.  We call it \emph{left exact} and \emph{additively cofibered} if, for each $X$-algebra $A$, the cofibered category $\sE(A)$ over $A\hMod$ has the corresponding property.  Note that this implies that each $\sE(A, I)$ possesses a zero object $e(A, I)$, well-defined up to unique isomorphism, and these can be combined into a cocartesian section $e$ of $\sE$ over $X\hAlgMod$.  We will call a stack of additively cofibered categories \emph{cartesian} if it is also cartesian over the arrows of $X\hAlgMod$ that are cartesian over $X\hAlg$.

Suppose that $\sE$ is a stack of additively cofibered categories on $X\hAlgMod$.  We call $\sE$ \emph{exact} if it is left exact and, for any $X$-algebra $A$, and any exact sequence
\begin{equation*}
0 \rightarrow I \rightarrow J \rightarrow K \rightarrow 0,
\end{equation*}
the map $\sE(A,J) \rightarrow \sE(A,K)$ is \emph{locally} essentially surjective.  This means that for any section $\xi$ of $\sE(A,K)$, the collection of $A$-algebras $B$ such that $\xi \tensor_A B$ can be lifted, up to isomorphism, to $\sE(B, J \tensor_A B)$ constitutes an \'etale cover of $A$.

Let $v : (A,I) \rightarrow (B, J)$ be a morphism of $X\hAlgMod$ that is cartesian over $X\hAlg$.  We have a morphism of $X$-algebras $u : A \rightarrow B$ and an induced isomorphism $I \simeq J_{[u]}$.  Then $v$ induces functors $\Phi : \sE(B,J) \rightarrow \sE(A,I)$ and $\Psi : \sE(A,I) \rightarrow \sE(B,J)$, each well-defined up to unique isomorphism, that must be \emph{adjoint}:  for each $x \in \sE(A,I)$ and $y \in \sE(B,J)$, we have
\begin{equation*}
\Hom_{\sE(A,I)}(x, \Phi(y)) = \Hom_f(x,y) = \Hom_{\sE(B,J)}(\Psi(x), y) .
\end{equation*}
But both $\sE(A,I)$ and $\sE(B,J)$ are groupoids by \cite[1.5~a)]{ccct}, so $\Phi$ and $\Psi$ must be mutually inverse functors.  
\begin{proposition} \label{prop:cart}
Let $\sE$ be a cofibered category over $X\hAlgMod$.  The following conditions are equivalent.
\begin{enumerate}[label=(\roman{*}), ref=(\roman{*})]
\item $\sE$ is cartesian over the morphisms of $X\hAlgMod$ that are cartesian over $X\hAlg$. \label{cart:1}
\item $\sE$ is cartesian over $X\hAlg$. \label{cart:4}
\item For every map $(A,I) \rightarrow (B,J)$ of $X\hAlgMod$ that is cartesian over $X\hAlg$, the functor $\sE(A,I) \rightarrow \sE(B,J)$ is an equivalence. \label{cart:2}
\item For every $X$-algebra homomorphism $f : A \rightarrow B$ and every $B$-module $J$, the functor $\sE(A,J_{[f]}) \rightarrow \sE(B,J)$ is an equivalence. \label{cart:3}
\end{enumerate}
\end{proposition}
\begin{proof}
The implication \ref{cart:1} \implies \ref{cart:2} was demonstrated before the statement of the proposition and \ref{cart:2} \implies \ref{cart:3} is obvious.  The equivalence \ref{cart:1} \iff \ref{cart:4} holds because $X\hAlgMod$ is cartesian over $X\hAlg$.


We prove \ref{cart:3} \implies \ref{cart:4}.  Let $f : A \rightarrow B$ be a homomorphism of $X$-algebras and let $J$ be a $B$-module.  Let $y$ be an object of $\sE(B,J)$.  Let $I = J_{[f]}$.  We show that there is an object $x \in \sE(A,J_{[f]})$ and a morphism $x \rightarrow y$ of $\sE$ that is cartesian above $X\hAlg$.  Since the co-base change functor $\Psi : \sE(A,I) \rightarrow \sE(B,J)$ is an equivalence, there is an object $x \in \sE(A,I)$ such that $\Psi(x)$ is isomorphic to $y$.  Choosing such an $x$ and such an isomorphism, we get an object of $\Hom_{\sE(B,J)}(\Psi(x), y) = \Hom_f(x,y)$, which gives the required map $\varphi : x \rightarrow y$.  It remains to demonstrate that $\varphi$ is cartesian over $X\hAlg$.

Let $g : C \rightarrow A$ be a homomorphism of $X$-algebras and suppose that $z \in \sE(C,K)$ for some $C$-module $K$.  We must show that the natural map
\begin{equation} \label{eqn:26}
\Hom_{g}(z,x) \rightarrow \Hom_{fg}(z,y)
\end{equation}
is a bijection.  Associated to any $\alpha \in \Hom_g(z,x)$ we have a map $(C,K) \rightarrow (B,J)$, hence a homomorphism of $C$-modules $K \rightarrow J_{[fg]} = I_{[f]}$.  Hence we can decompose~\eqref{eqn:26} as
\begin{equation*}
\coprod_{u : K \rightarrow I_{[f]}} \Hom_{(g,u)}(z,x) \rightarrow \coprod_{u : K \rightarrow J_{[fg]}} \Hom_{(fg, u)}(z,y) .
\end{equation*}
It suffices to demonstrate that the map is a bijection on each component.  Now, because $\sE$ is cocartesian over $X\hAlgMod$, there must be a morphism $z \rightarrow w$ of $\sE$ that is cocartesian over the map $(g,u) : (C,K) \rightarrow (C,I_{[f]})$.  The horizontal arrows in the commutative diagram
\begin{equation*} \xymatrix{
\Hom_{(g,\id_{I_{[f]}})}(w,x) \ar[r] \ar[d] & \Hom_{(g,u)}(z,x) \ar[d] \\
\Hom_{(fg,\id_{I_{[f]}})}(w,y) \ar[r] &  \Hom_{(fg,u)}(z,y)
}\end{equation*}
are both bijections, so it will now suffice to demonstrate that the vertical arrow on the left is a bijection.  

Use the notation
\begin{gather*}
\Psi_f : \sE(A,I) \rightarrow \sE(B,J) \\
\Psi_g : \sE(C,I_{[f]}) \rightarrow \sE(A,I) \\
\end{gather*}
for the co-base change functors.  We have
\begin{align*}
\Hom_{fg}(w,y) & = \Hom_{\sE(B,J)}(\Psi_{f} \Psi_g(w), y) & & \text{because $w \rightarrow \Psi_f \Psi_g(w)$ is cocartesian} \\
& = \Hom_{\sE(A,I)}(\Psi_{g} (w), x) & & \text{because $\Psi_f$ is an equivalence} \\ & & & \qquad \text{and $\Psi_f(x) \rightarrow y$ is an isomorphism} \\
& = \Hom_g(w,x) & & \text{because $w \rightarrow \Psi_g(w)$ is cocartesian} .
\end{align*}
\end{proof}

\begin{definition} \label{def:obs}
Let $h : X \rightarrow Y$ be a morphism of homogeneous stacks over $\ComRng$.  An obstruction theory for $X$ over $Y$ is a cartesian stack of left exact additively cofibered categories $\sE$ over $X$ equipped with a cartesian diagram
\begin{equation} \label{eqn:6}\xymatrix{
  \Exal_X \ar[r] \ar[d] & e \ar[d] \\
  h^\ast \Exal_Y \ar[r] & \sE
} \end{equation}
in which the horizontal arrows preserve
\begin{enumerate}
\item arrows that are cocartesian over \'etale maps in $\ComRngMod$, and
\item arrows that are cartesian over $\ComRng$.
\end{enumerate}
\end{definition}

\subsection{The relative intrinsic normal stack} \label{sec:ins}

To specify Diagram~\eqref{eqn:6} should be the same as to give a morphism
\begin{equation*}
h^\ast \Exal_Y / \Exal_X \rightarrow \sE ,
\end{equation*}
provided one can give a satisfactory definition of the source.  We indicate how this can be done when $X$ and $Y$ are algebraic stacks and $X \rightarrow Y$ is of \emph{Deligne--Mumford type}, meaning that for any $Y$-algbra $A$, the stack $X_A$ on $A$-algebras is a Deligne--Mumford stack.

We make some abbreviations in this section.  If $A$ is an $X$-algebra, then we write $\Spec A$ for the spectrum of its underlying commutative ring, equipped with a map to the stack $X$.  We also write $\et(A)$ in place of $\et(\Spec A)$ for the small \'etale site of $\Spec A$.  If $J$ is an $A$-module, we write $\tJ$ for the associated sheaf on $\et(A)$.  Thus $\tA$ is the structure sheaf on $\et(A)$.


Assume first that both $X$ and $Y$ are Deligne--Mumford stacks.  Then the map $h : X \rightarrow Y$ can be represented by a morphism of \'etale locally ringed topoi $(h, \varphi) : (\et(X), \cO_X) \rightarrow (\et(Y), \cO_Y)$.  
\begin{definition} \label{def:ins}
Suppose that $X \rightarrow Y$ is a morphism of Deligne--Mumford stacks.  Define $\sN_{X/Y}$ to be the category of triples $(C, J, \sA)$ where $(C,J) \in X\hAlgMod$, corresponding to a morphism $f : \et(C) \rightarrow \et(X)$ and a homomorphism $\varphi : f^\ast \cO_X \rightarrow \tC$ of sheaves of rings, and $\sA$ is a square-zero extension of $f^\ast \cO_X$ by $\tJ_{[\varphi]}$ as a $f^\ast h^\ast \cO_Y$-algebra.
\end{definition}



Let $X\hEtAlg$ be the category of \'etale $X$-algebras.  Then we can restrict the projections $\sN_{X/Y} \rightarrow X\hAlg$ and $h^\ast \Exal_Y \rightarrow X\hAlg$ via the inclusion $X\hEtAlg \rightarrow X\hAlg$.  
\begin{lemma}
The restrictions of $\sN_{X/Y}$ and $h^\ast \Exal_Y$ to $X\hEtAlg$ are equivalent.
\end{lemma}
\begin{proof}
Suppose that $C$ is an \'etale $X$-algebra.  Let $i : \et(C) \rightarrow \et(X)$ be the canonical map.  Then the sheaves $\tC$ and $i^\ast \cO_X$ on $\et(C)$ are canonically isomorphic.  Moreover, if $J$ is a $C$-module, then extensions of $\tC$ by $\tJ$ as a $i^\ast h^\ast \cO_Y$-algebra are equivalent by taking global sections to extensions of $C$ by $J$ as a $Y$-algebra.
\end{proof}

Let $(B, I, \sB)$ and $(C, J, \sC)$ be objects of $\sN_{X/Y}$ with $u : \et( B) \rightarrow \et(X)$ and $v : \et( C) \rightarrow \et(X)$ the corresponding morphisms of \'etale sites.  A morphism $(B, I, \sB) \rightarrow (C, J, \sC)$ of $\sN_{X/Y}$ is a morphism $f : (B,I) \rightarrow (C,J)$ of $X\hAlgMod$, corresponding to a map $w : \et( C) \rightarrow \et( B)$, and a morphism of extensions $w^\ast \sC \rightarrow \sB$ that is compatible with the $Y$-algebra structures and the maps $w^\ast \tC \rightarrow \tB$ and $w^\ast \tJ \rightarrow \tI$ induced from $f$.\marg{is it clear enough what that means?}

There is a projection $\sN_{X/Y} \rightarrow X\hAlgMod$ sending a triple $(C, J, \sA)$  to $(C,J)$, as in Definition~\ref{def:ins}.


\begin{proposition}
\begin{enumerate}[label=(\roman{*}), ref=(\roman{*})]
\item \label{ins:1} $\sN_{X/Y}$ is cocartesian over $X\hAlgMod$.
\item \label{ins:2} $\sN_{X/Y}$ is cartesian over $X\hAlg$ and the projection $\sN_{X/Y} \rightarrow X\hAlgMod$ preserves morphisms that are cartesian over $X\hAlg$.
\item \label{ins:3} For each $X$-algebra $C$, the cofibered category $\sN_{X/Y}(C)$ over $C\hMod$ is additively cofibered and left exact.
\end{enumerate}
\end{proposition}
\begin{proof}
For the first two assertions we limit ourselves to describing the pushforward and pullback functors.  Their universal properties are not difficult to verify.

\ref{ins:1}
Suppose that $(B, I) \rightarrow (C, J)$ is a morphism in $X\hAlgMod$ whose corresponding morphism of \'etale sites is $w : \et(C) \rightarrow \et(B)$ and $(B,I,\sB) \in \sN_{X/Y}(B,I)$.  Let $u : \et(B) \rightarrow \et(X)$ and $v : \et(C) \rightarrow \et(X)$ be the morphisms of \'etale sites.  Then $w^\ast$ is exact so $w^\ast \sB$ is an extension of $w^\ast u^\ast \cO_X \simeq v^\ast \cO_X$ by $w^\ast \tI$.  Pushing out this extension via the map $w^\ast \tI \rightarrow \tJ$, we get the desired extension of $v^\ast \cO_X$ by $\tJ$.

\ref{ins:2}
Suppose $\sC \in \sN_{X/Y}(C)$ for some $C \in X\hAlg$, and $u : \et(C) \rightarrow \et(X)$ is the corresponding map of \'etale sites.  Let $\varphi : B \rightarrow C$ be an $X$-algebra homomorphism and let $w : \et(C) \rightarrow \et(B)$ be the corresponding map of \'etale sites.  Write $v$ for the map of \'etale sites $\et(B) \rightarrow \et(X)$.  Then $\sC$ is an extension of $u^\ast \cO_X$ by $\tJ$ for some $C$-module $J$.  Using the fact that $w$ is affine, so $R^1 w_\ast \tJ = 0$, it follows that $w_\ast \sC$ is an extension of $w_\ast u^\ast \cO_X$ by $w_\ast \tJ = \tJ_{[\varphi]}$.  Pulling this extension back via the map $v^\ast \cO_X \rightarrow w_\ast u^\ast \cO_X$ gives the extension of $v^\ast \cO_X$ by $\tJ_{[\varphi]}$ that we seek.

\ref{ins:3}
The proof is very similar to that of Proposition~\ref{prop:def-cof}~\ref{exal:3} and we omit it.  A proof of the left exactness may be found in \cite[Section~7.1.11]{ccct}.
%
%
%
\end{proof}

Now we construct a cartesian diagram~\eqref{eqn:6} (with $\sE = \sN_{X/Y}$), demonstrating that $\sN_{X/Y}$ is an obstruction theory for $X$ over $Y$.  First we construct the map $h^\ast \Exal_Y \rightarrow \sN_{X/Y}$.  An object of $h^\ast \Exal_Y$ is representable by an $X$-algebra $C$ and a square-zero extension $C'$ of $C$ as a $Y$-algebra.  Let $I$ denote the ideal of $C$ in $C'$ and  let $f$ denote the map of \'etale topoi $\et( C) \rightarrow X$ coming from the $X$-algebra structure on $C$.  Since the \'etale site of $C'$ is the same as that of $C$, we obtain a square-zero extension $\tC'$ of $\tC$ by $\tI$ on $\et( C)$ as a $f^\ast h^\ast \cO_Y$-algebra.  Pulling back via the map $f^\ast \cO_X \rightarrow \tC$, we get a square-zero extension $\sC = \tC' \fp_{\tC} f^\ast \cO_X$ of $f^\ast \cO_X$ on $\et(C)$ as a $f^\ast h^\ast \cO_Y$-algebra:
\begin{equation}  \label{eqn:24} \xymatrix{
0 \ar[r] & \tI \ar[r] \ar[d] & \sC \ar[r] \ar[d] & f^\ast \cO_X \ar[r] \ar[d] & 0 \\
0 \ar[r] & \tI \ar[r] & \tC' \ar[r] & \tC \ar[r] & 0.
} \end{equation}
This gives an object of $\sN_{X/Y}(C, I)$.

To verify that Diagram~\eqref{eqn:6} is cartesian, we check that isomorphisms between the object constructed above and the trivial extension of $f^\ast \cO_X$ by $\tI$ are in bijection with the $X$-algebra structures on $\tC'$ lifting the given one on $\tC$.  But to give an $X$-algebra structure on $\tC'$ is the same as to lift the diagram
\begin{equation*} \xymatrix{
& & & f^\ast \cO_X \ar@{-->}[dl] \ar[d] \\
0 \ar[r] & \tI \ar[r] & \tC' \ar[r] & \tC \ar[r] & 0
} \end{equation*}
which is of course the same as to split the extension in the first row of Diagram~\eqref{eqn:24}.

\begin{proposition} \label{prop:ins}
Let $\sE$ be an obstruction theory for $X$ over $Y$.  Then there is a uniquely determined (up to unique isomorphism) fully faithful functor $\sN_{X/Y} \rightarrow \sE$ inducing the cartesian diagram~\eqref{eqn:6} from the corresponding diagram for $\sN_{X/Y}$ (up to a unique isomorphism).
\end{proposition}
\begin{proof}
Since both $\sE$ and $\sN_{X/Y}$ are stacks on $X\hAlg$, it is sufficient to describe the value of the map $\sN_{X/Y} \rightarrow \sE$ on $(C,J)$ \'etale locally in $C$.  Since $X$ is a Deligne--Mumford stack, this means that we can assume there is an \'etale $X$-algebra $D$ and an $X$-algebra homomorphism $\varphi : D \rightarrow C$.  Let $g : \et(C) \rightarrow \et(D)$ be the induced morphism of \'etale sites.  The pushforward maps 
\begin{gather*}
\sN_{X/Y}(D, J_{[\varphi]}) \rightarrow \sN_{X/Y}(C, J) \\
\sE(D, J_{[\varphi]}) \rightarrow \sN_{X/Y}(C, J)
\end{gather*}
are both equivalences by Proposition~\ref{prop:cart}.  Moreover, if the map $D \rightarrow C$ factors through a map $\psi : D' \rightarrow C$, where $D'$ is another \'etale $X$-algebra, then the diagrams
\setlength{\hoffset}{-0.5in}
\begin{equation*} \xymatrix@C=0pt{
\sF_{X/Y}(D, J_{[\varphi]}) \ar[rr] \ar[dr] & & \sF_{X/Y}(D', J_{[\psi]}) \ar[dl] \\
& \sF_{X/Y}(C,J) 
}  \end{equation*}
(with $\sF = \sN$ or $\sF = \sE$) commute in a canonical way, and a similar statement holds for the tetrahedra coming from a sequence $D \rightarrow D' \rightarrow D'' \rightarrow C$.  Therefore, to define the map over all $X$-algebras $C$, it is sufficient to provide a definition when $f : \et(C) \rightarrow X$ is \'etale.

In this case, the map $f^\ast \cO_X \rightarrow \tC$ is an isomorphism, so $\sN_{X/Y}(C) = h^\ast \Exal_Y(C)$ and we have been given a map $h^\ast \Exal_Y(C) \rightarrow \sE(C)$ by hypothesis.

To check that the cartesian diagram~\eqref{eqn:6} can be recovered from the one for $\sN_{X/Y}$, it is sufficient to verify this over all \'etale $X$-algebras, in which case it is true by definition.
\end{proof}

The proposition demonstrates that $\sN_{X/Y}$ is the \emph{universal} obstruction theory for $X$ over $Y$.

We can extend this construction to the situation where $Y$ is an algebraic stack and $X$ is only assumed to be of Deligne--Mumford type over $Y$.  As before, it is enough to define $\sN_{X/Y}(C, J)$ \'etale locally in $C$, so we can assume that there is a smooth $Y$-algebra $B$ and a $Y$-algebra homomorphism $B \rightarrow C$.  We then take $\sN_{X/Y}(C,J) = \sN_{X_B/B}(C,J)$.  To demonstrate that this is well-defined, we must prove
\begin{lemma}[cf.\ {\cite[Lemma~5.3]{def2}}]
If $A \rightarrow B$ is a flat homomorphism of commutative rings, $X_A$ is a Deligne--Mumford stack over $A$, and $X_B = X \tensor_A B$, then for any $X_B$-algebra $C$ and any $C$-module $J$, the natural map $\sN_{X_B/B}(C,J) \rightarrow \sN_{X_A/A}(C,J)$ is an equivalence.
\end{lemma}
\begin{proof}
We describe an inverse.  Let $f : \et(C) \rightarrow \et(X_B)$ and $g : \et(C) \rightarrow \et(X_A)$ be the maps induced from the $X_A$- and $X_B$-algebra structures on $C$.  Suppose that $J$ is a $C$-module and $\sC$ is an extension of $g^\ast \cO_{X_A}$ by $\tJ$.  Then $\sC \tensor_A B$ is an extension of $g^\ast \cO_{X_A} \tensor_A B = f^\ast \cO_{X_B}$ by $\tJ \tensor_A B$.  There is a canonical map $\tJ \tensor_A B \rightarrow \tJ$ coming from the bilinear map $B \times J \rightarrow J$ induced from the $C$-action on $J$ and the map $B \rightarrow C$.  Pushing out the extension by this map gives the required extension of $f^\ast \cO_{X_B}$ by $\tJ$.
\end{proof}
The verification that this is an obstruction theory for $X$ over $Y$ is a local problem and reduces to the verification that $\sN_{X_B/B}$ is an obstruction theory for $X_B$ over $B$, which was shown above.

\begin{corollary} \label{cor:ins}
An obstruction theory for a morphism $X \rightarrow Y$ of Deligne--Mumford type may be specified by giving a cartesian, left-exact, additively cofibered category $\sE$ over $X\hAlgMod$ and a fully faithful morphism $\sN_{X/Y} \rightarrow \sE$ respecting morphisms that are cocartesian over $X\hAlgMod$ and morphisms that are cartesian over $X\hAlg$.
\end{corollary}

\begin{remark}
Given a suitable definition of a $2$-category of obstruction, the corollary may be rephrased to say that the $2$-category of obstruction theories is equivalent to the $2$-category of fully faithful morphisms $\sN_{X/Y} \rightarrow \sE$.
\end{remark}

\subsection{Global obstructions} \label{sec:global}

Let $h : X \rightarrow Y$ be a morphism of homogeneous stacks over $\ComRng$ and suppose that $\sE$ is an obstruction theory for $X$ over $Y$.  Consider a lifting problem
\begin{equation} \label{eqn:4} \xymatrix{
  S \ar[r] \ar[d] & X \ar[d]^h \\
  S' \ar[r] \ar@{-->}[ur] & Y
} \end{equation}
in which $S$ and $S'$ are algebraic spaces and $S'$ is a square-zero extension of $S$.  This implies that the ideal $J$ of $S$ in $S'$ is a quasi-coherent $\cO_S$-module on the small \'etale site of $S$ (which we identify with the small \'etale site of $S'$).  

Let $\usE$ be the stack on the small \'etale site of $S$ whose category of sections over $U$ is $\sE(U, J_U)^\circ$ when $U$ is affine and \'etale over $S$.  This is a legitimate definition since $\sE$ was assumed to be a stack in the \'etale topology.  If it becomes necessary to specify the dependence of $\usE$ on $S$ and $J$, we will write $\usE(S, J)$.

Define $\uDef_X(S,\sI)$ to be the stack on the small \'etale site of $S$ whose sections are commutative diagrams
\begin{equation*} \xymatrix{
S \ar[r] \ar[d] & X \\
S' \ar[ur]
} \end{equation*}
where $S'$ is a square-zero extension of $S$ by the quasi-coherent sheaf of ideals $\sI$.  Defining $\uDef_Y$ likewise, so that $h^\ast \uDef_Y(S,\sI)$ is the collection of commutative diagrams
\begin{equation*} \xymatrix{
S \ar[r] \ar[d] & X \ar[d] \\
S' \ar[r] & Y
} \end{equation*}
where $S'$ is a square-zero extension of $S$ by $\sI$, we obtain a map $\uDef_X(S,\sI) \rightarrow h^\ast \uDef_Y(S,\sI)$.  Note that if $S = \Spec A$ is affine, we have an equivalence
\begin{gather*}
\uDef_X(S, \tI) = \Exal_X(A, I)^\circ \\
h^\ast \uDef_Y(S, \tI) = h^\ast \Exal_Y(A, I)^\circ
\end{gather*}
and these equivalences are functorial in $A$.  When $S$ and $\sI$ are fixed, we omit them from the notation and write $\uDef_X$ for $\uDef_X(S, \sI)$.

Since $\sE$ is an obstruction theory and $\uDef_X$, $\uDef_Y$, and $\usE$ are all stacks, we have a cartesian diagram
\begin{equation} \label{eqn:5} \xymatrix{
  \uDef_X \ar[d] \ar[r] & e \ar[d] \\
  h^\ast \uDef_Y \ar[r]^\omega & \usE .
} \end{equation}
The commutativity of the diagram above induces a map
\begin{equation*}
  \uDef_X \rightarrow \set[(\xi, \varphi)]{\xi \in \uDef_Y, \varphi : \omega(\xi) \simeq e}
\end{equation*}
from $\uDef_X$ to the category of pairs $(\xi, \varphi)$ where $\xi$ is a section of $\uDef_Y$ and $\varphi$ is an isomorphism between the image $\omega(\xi)$ in $\usE$ and the zero section $e$ of $\usE$.  Note now that $h^\ast \uDef_Y$ can be identified with the collection of diagrams of solid lines~\eqref{eqn:4} and that to solve the lifting problem indicated in~\eqref{eqn:4} is precisely to lift a section of $h^\ast \uDef_Y$ to one of $\uDef_X$.  The cartesian diagram~\eqref{eqn:5} then provides us with a bijection between lifts of $\xi \in h^\ast \uDef_Y$ and isomorphisms between $\omega(\xi)$ and $e$.  In other words, we may think of $\omega(\xi)$ as the obstruction to the existence of a lift of $\xi$:  it is isomorphic to $e$ if and only if a lift exists and such isomorphisms correspond exactly to the lifts.

\section{Representability} \label{sec:rep}

\subsection{Additively cofibered categories associated to complexes}

Let $\sC$ be an abelian category.  Suppose that $\bF_\bullet$ is a chain complex.  Following \cite[Section~3.1]{biext1}, we define an additively cofibered category $\Psi_{\bF_\bullet}$ of diagrams
\begin{equation*} \xymatrix{
& & & \bF_2 \ar[d] \ar@/_10pt/[ddl]_0 \\
& & & \bF_1 \ar[d] \ar[dl] \\
0 \ar[r] & J \ar[r] & X \ar[r] & \bF_0 \ar[r] & 0
} \end{equation*}
in which the sequence at the bottom is exact.  A morphism in $\Psi_{\bF_\bullet}$ from $(X, I, \alpha) \in \Psi_{\bF_\bullet}(I)$ to $(Y,J,\beta) \in \Psi_{\bF_\bullet}(J)$ is a commutative diagram
\begin{equation*} \xymatrix@R=10pt{
0 \ar[r] & I \ar[r] \ar[d] & X \ar[r] \ar[d] & \bF_0 \ar@{=}[d] \ar[r] & 0 \\
0 \ar[r] & J \ar[r] & Y \ar[r] & \bF_0 \ar[r] & 0
} \end{equation*}
such that the map $X \rightarrow Y$ carries $\alpha$ to $\beta$.

Objects of $\Psi_{\bF_\bullet}(J)$ are called \emph{extensions of $\bF_\bullet$ by $J$}.  We also write $\Ext(\bF_\bullet, J)$ for $\Psi_{\bF_\bullet}(J)$.

\begin{example}
\begin{enumerate}[label=(\roman{*})]
\item If $\bF_\bullet = \bF_1[1]$, we have $\Psi_{\bF_\bullet}(J) = \Hom(\bF_\bullet, J)$.
\item If $\bF = \bF_0[0]$ then $\Psi_{\bF_\bullet}(J)$ is the category of extensions of $\bF_0$ by $J$.
\end{enumerate}
\end{example}

\begin{definition} \label{def:rep-loc}
An additively cofibered category $\sE$ over an abelian category $\sC$ is said to be \emph{representable} if there exists a $2$-term complex $\bF_\bullet$ in $\sC$, concentrated in degrees $[-1,0]$ and an isomorphism $\sE \simeq \Psi_{\bF_\bullet}$.
\end{definition}

\begin{remark}
Note that if $\bE_\bullet$ is any chain complex, $\Psi_{\bE_\bullet} = \Psi_{\tau_{\geq -1} \bE_\bullet}$, so the definition is equivalent to the one that would be obtained by suppressing the restriction on the degrees in which $\bE_\bullet$ is concentrated.
\end{remark}

The definition of $\Psi_{\bF_\bullet}$ can be extended to cochain complexes.  Suppose that $\bE^\bullet$ is a cochain complex (in non-negative degrees).  Following \cite[Section~3.3.2]{biext1}, we define $\Psi_{\bF_\bullet}(\bE^\bullet)$ to be the category of diagrams
\begin{equation*} \xymatrix{
& & & \bF_2 \ar[d] \ar@/_10pt/[ddl]_0  \\
& & & \bF_1 \ar[d] \ar[dl] \\
0 \ar[r] & \bE^0 \ar[r] \ar[d] & X \ar[r] \ar[dl] \ar@/^10pt/[ddl]^0  & \bF_0 \ar[r] & 0 \\
& \bE^1  \ar[d] \\
& \bE^2 
} \end{equation*}
in which the middle row is exact and the maps
\begin{gather*}
\bF_1 \rightarrow X \rightarrow \bE^1 \\
\bF_2 \rightarrow X \\
X \rightarrow \bE^2  
\end{gather*}
are all zero.  We note that these are abelian examples of \emph{butterflies} \cite{Ald,AN1,AN2}.

The construction above is a special case of a more general construction.  Suppose that $\sE$ is an additively cofibered category over an abelian category $\sC$.  Let $\bF^\bullet$ be a cochain complex.  Define $\sE(\bF^\bullet)$ to be the category of pairs $(X, \alpha)$ where $X \in \sE(\bF^0)$ and $\alpha : d_\ast X \simeq 0_{\sE(\bF^1)}$ is an isomorphism inducing the zero isomorphism
\begin{equation*}
0_{\sE(\bF^2)} \simeq d_\ast d_\ast X \xrightarrow{d_\ast \alpha} d_\ast 0_{\sE(\bF^1)} \simeq 0_{\sE(\bF^2)} .
\end{equation*}
This determines an additively cofibered category over the category of cochain complexes.  Moreover, if $\sE$ is a left exact and $\bF^\bullet \rightarrow \bG^\bullet$ is a quasi-isomorphism then by \cite[Proposition~2.7]{ccct}, the induced map $\sE(\bF^\bullet) \rightarrow \sE(\bG^\bullet)$ is an equivalence.

Suppose that $X$ is a stack in groupoids over $\ComRng$ and $\bF_\bullet$ is a chain complex of sheaves of flat modules\footnote{In fact, the definition and many of the results are valid when only $\bF_0$ is required to be flat.} on $X\hAlg$.  That is, for each $X$-algebra $A$, we have a chain complex $\bF_\bullet(A)$ of sheaves of $\tA$-modules on $\et(A)$, concentrated in degrees $\leq 0$, with $\bF_i(A)$ flat over $\tA$, equipped with maps $u^\ast \bF_\bullet(A) \rightarrow \bF_\bullet(B)$ for each morphism $A \rightarrow B$ of $X$-algebras corresponding to $u : \et(B) \rightarrow \et(A)$.  These maps are required to satisfy the expected compatibility condition and to induce isomorphisms $u^\ast \bF_\bullet(A) \tensor_{u^\ast \tA} \tB \xrightarrow{\sim} \bF_\bullet(B)$.  Define $\Psi_{\bF_\bullet}$ to be the category of triples $(A, J, X)$ where $A$ is an $X$-algebra, $J$ is an $A$-module, and $X$ is an extension of $\bF_\bullet(A)$ by~$\tJ$.  If $\bE^\bullet$ is a chain complex of $A$-modules then we can define $\Psi_{\bF_\bullet}(A, \bE^\bullet) = \Psi_{\bF_\bullet(A)}(\tbE^\bullet)$.

\begin{lemma}
Suppose that $X$ is a stack in groupoids on $\ComRng$.  Let $\bF_\bullet$ be a chain complex of flat modules on $X$.  The category $\Psi_{\bF_\bullet}$ defined above is a left exact, additively cofibered, cartesian stack over $X$.
\end{lemma}
\begin{proof}
We describe the pushforward functor associated to morphisms of $X\hAlgMod$.  Let $Z$ be an extension of $\bF_\bullet(A)$ by $\tI$ for some $X$-algebra $A$ and $A$-module $I$.  Let $(A,I) \rightarrow (B,J)$ be a homomorphism in $X\hAlgMod$ and let $u : \et(B) \rightarrow \et(A)$ be the corresponding morphism of \'etale sites.  Then $u^\ast Z \tensor_{u^\ast \tA} \tB$ is an extension of $u^\ast \bF_\bullet(A) \tensor_{\tA} \tB = \bF_\bullet(B)$ by $u^\ast \tI \tensor_{u^\ast \tA} \tB$ (by the flatness of $\bF(A)_0$).  Pushing out this extension by the map $u^\ast \tI \tensor_{u^\ast \tA} \tB \rightarrow \tJ$, we get the extension we are looking for.

We also check that $\Psi_{\bF_\bullet}$ is fibered over $X\hAlg$.  Let $\varphi : A \rightarrow B$ be a morphism of $X$-algebras and $u : \et(B) \rightarrow \et(A)$ the corresponding morphism of \'etale sites.  Suppose that $Z$ is an extension of $\bF_\bullet(B)$ by $\tJ$ for some $B$-module $J$.  Then since $\tJ$ is quasi-coherent, $u_\ast Z$ is an extension of $u_\ast \bF_\bullet(B)$ by $u_\ast \tJ = \tJ_{[\varphi]}$.  Pulling this extension back via the map $\bF_\bullet(A) \rightarrow u_\ast \bF_\bullet(B)$ gives the desired object of $\Psi_{\bF_\bullet}(A,J_{[\varphi]})$.

To see that $\Psi_{\bF_\bullet}$ forms a stack, it's enough to remark that $\bF_\bullet$ forms an \'etale sheaf of complexes of $\cO_X$-modules and extensions of $\cO_X$-modules satisfy \'etale descent.  The left exactness is well-known and we omit the proof.
\end{proof}

\begin{definition} \label{def:rep-glob}
If $\sE$ is an additively cofibered stack over $X$ then we call $\sE$ \emph{representable} if it is isomorphic to a stack of the form $\Psi_{\bF_\bullet}$ for some chain complex of flat modules on $X$.
\end{definition}

Note that if $\sE$ is represented by the complex $\bF_\bullet$ in the sense of Definition~\ref{def:rep-glob} then for each $X$-algebra $A$, the additively cofibered category $\sE(A)$ on $A\hMod$ is represented by $\bF_\bullet(A)$, in the sense of Definition~\ref{def:rep-loc}.

\subsubsection*{Morphisms of complexes and additively cofibered categories}

Let $\bE_\bullet \rightarrow \bF_\bullet$ be a morphism of $2$-term complexes.  This induces a map $\Psi_{\bF_\bullet} \rightarrow \Psi_{\bE_\bullet}$ by pullback.

The following lemma says that the functor which takes a complex to its associated additively cofibered category is $2$-categorially fully faithful.  It is essentially an example of the $2$-categorical Yoneda lemma.

\begin{lemma} \label{lem:morphs}
Let $\bE_\bullet$ and $\bF_\bullet$ be $2$-term chain complexes in degrees $[-1,0]$ in an abelian category $\sC$.  Let $\sE = \Psi_{\bE_\bullet}$ and $\sF = \Psi_{\bF_\bullet}$ be the associated additively cofibered categories over $\sC$.  
\begin{enumerate}[label=(\roman{*}), ref=Lemma~\ref{lem:morphs}~(\roman{*})]
\item \label{lem:adj} There is an equivalence of categories
\begin{equation*}
\Hom(\sF, \sE) \simeq \Psi_{\bE_\bullet}(\bF_\bullet[-1]).
\end{equation*}
\item There is a bijection
\begin{equation*}
\Hom_{D(\sC)}(\bE_\bullet, \bF_\bullet) \rightarrow \Hom(\sF, \sE) / \text{isom.}
\end{equation*}
\item \label{lem:psi-equiv} If $\bE_\bullet \rightarrow \bF_\bullet$ is a quasi-isomorphism then the induced map $\Psi_{\bF_\bullet} \rightarrow \Psi_{\bE_\bullet}$ is an equivalence.
\end{enumerate}
\end{lemma}
\begin{proof}
See \cite[Corollaire~6.13]{ccct}.
\end{proof}

Let $u, v : \bE_\bullet \rightarrow \bF_\bullet$ be two homomorphisms of complexes in an abelian category $\sC$ and let $h$ be a chain homotopy from $u$ to $v$.  Then $h$ induces an isomorphism between the two induced maps $\Psi_{\bF_\bullet} \rightarrow \Psi_{\bE_\bullet}$.  As an object of $\Psi_{\bE_\bullet}(\bF_\bullet[-1])$, the map $u$ can be described by a diagram
\begin{equation*} \xymatrix{
& & & \bE_1 \ar[d]^d \ar[dl]_{\bigl( \begin{smallmatrix} d \\ u_1 \end{smallmatrix} \bigr)} \\
0 \ar[r] & \bF_1 \ar[d]_{-d} \ar[r]^{\bigl( \begin{smallmatrix} 0 \\ -\id \end{smallmatrix} \bigr)} & \bE_0 \times \bF_1 \ar[dl]^{( \begin{smallmatrix} u_0 & d \end{smallmatrix} )} \ar[r]_{( \begin{smallmatrix} \id & 0 \end{smallmatrix} )} & \bE_0 \ar[r] & 0 . \\
& \bF_0 
} \end{equation*}
Note the sign on the differential $\bF_1 \rightarrow \bF_0$ because of the shift.  There is of course a similar diagram with $u$ replaced by $v$ describing the object of $\Psi_{\bE_\bullet}(\bF_\bullet[-1])$ induced from $v$.  Let $\mu$ and $\nu$ be the corresponding objects of $\Psi_{\bE_\bullet}(\bF_\bullet[-1])$.  The chain homotopy $h$ gives us a map
\begin{equation*}
\bigl( \begin{smallmatrix} \id & 0 \\ h & \id \end{smallmatrix} \bigr) : \bE_0 \times \bF_1 \rightarrow \bE_0 \times \bF_1 .
\end{equation*}
One checks that this gives a morphism of commutative diagrams, and therefore gives an isomorphism between $\mu$ and $\nu$.  It follows immediately from the definition that the composition of chain homotopies (by addition) is carried by this construction to the composition of morphisms in $\Psi_{\bE_\bullet}(\bF_\bullet[-1])$.

\subsection{The cotangent complex} \label{sec:cc}

Assume that $X \rightarrow Y$ is of Deligne--Mumford type and $Y$ is algebraic.  There is therefore a relative cotangent complex $\bL_{X/Y}$ and we can identify $\sN_{X/Y}(C, J) = \Ext(f^\ast \bL_{X/Y}, J)$,\marg{reference} where $f : \et(C) \rightarrow \et(X)$ is the morphism of \'etale sites induced from the $X$-algebra structure on $C$.  We likewise have $\Exal_X(C, f, J) = \Ext(\bL_{C/X}, J)$ and $h^\ast \Exal_Y(C, f, J) = \Ext(\bL_{C/Y}, J)$.  Moreover, the exact sequence
\begin{equation*}
0 \rightarrow \Exal_X(C, J) \rightarrow h^\ast \Exal_Y(C, J) \rightarrow \sN_{X/Y}(C, J)
\end{equation*}
coincides via these identifications with the exact sequence
\begin{equation*}
0 \rightarrow \Ext(\bL_{C/X}, J) \rightarrow \Ext(\bL_{C/Y}, J) \rightarrow \Ext(f^\ast \bL_{X/Y}, J)
\end{equation*}
associated to the exact triangle 
\begin{equation*}
f^\ast \bL_{X/Y} \rightarrow \bL_{C/Y} \rightarrow \bL_{C/X} \rightarrow f^\ast \bL_{X/Y}[1] .
\end{equation*}

\begin{proposition}
Suppose $\bE_\bullet \rightarrow \bL_{X/Y}$ is a morphism of chain complexes that induces an isomorphism $H_0(\bE_\bullet) \rightarrow H_0(\bL_{X/Y})$ and a surjection $H_1(\bE_\bullet) \rightarrow H_1(\bL_{X/Y})$.  Then the induced map $\sN_{X/Y} \rightarrow \sE$, where $\sE = \Ext(\bE_\bullet, -)$, is an obstruction theory.
\end{proposition}
\begin{proof}
We only need to check that it is fully faithful.  Since these are cofibered categories whose fibers are abelian $2$-categories, it's enough to check that the map induces an isomorphism on the automorphism group of the identity elements and is injective on isomorphism classes.  But the map on automorphism groups of the identity elements (resp.\ on isomorphism classes) can respectively be identified with the maps
\begin{gather*}
u_p : \Ext^p(\bL_{X/Y}, J) \rightarrow \Ext^p(\bE_\bullet, J)  .
\end{gather*}
for $p = 0$ (resp.\ for $p = 1$).  Let $\bF_\bullet$ be the cone of $\bE_\bullet \rightarrow \bL_{X/Y}$.  This is quasi-isomorphic to a chain complex concentrated in degrees $2$ and higher.  Therefore $\Hom(\bF_\bullet, J) = \Ext^1(\bF_\bullet, J) = 0$, so by the long exact sequence
\begin{multline*}
\Hom(\bF_\bullet, J) \rightarrow \Hom(\bL_{X/Y}, J) \rightarrow \Hom(\bE_\bullet, J) \rightarrow \\ \rightarrow \Ext^1(\bF_\bullet, J)  \rightarrow
\Ext^1(\bL_{X/Y}, J) \rightarrow \Ext^1(\bE_\bullet, J)
\end{multline*}
we deduce that $u_p$ is an isomorphism for $p = 0$ and injective for $p = 1$.
\end{proof}

\begin{corollary} \label{cor:BF1}
Any obstruction theory in the sense of Behrend--Fantechi gives rise to an obstruction theory in the sense of Definition~\ref{def:obs}.
\end{corollary}

\begin{corollary}[of Lemma~\ref{lem:morphs}] \label{cor:BF2}
Suppose that $X \rightarrow Y$ is a morphism of Deligne--Mumford type and $\sN_{X/Y} \rightarrow \sE$ is an obstruction theory for $X$ over $Y$.  If the additively cofibered stack $\sE$ is representable by a flat chain complex $\bE_\bullet$ on $X$ with coherent cohomology then the obstruction theory $\sN_{X/Y} \rightarrow \sE$ is induced from an obstruction theory $\bE_\bullet \rightarrow \bL_{X/Y}$ in the sense of Behrend--Fantechi.
\end{corollary}
\begin{proof}
Lemma~\ref{lem:morphs} guarantees that the map $\sN_{X/Y} \rightarrow \sE$ is induced from a map $\bE_\bullet \rightarrow \bL_{X/Y}$ where $\bE_\bullet$ is concentrated in degrees $\leq 0$.  Since we have assumed the cohomology of $\bE_\bullet$ is coherent, $\bE_\bullet$ satisfies Condition~($\star$) of \cite[Definition~2.3]{BF}.

We must check that the map $\bE_\bullet \rightarrow \bL_{X/Y}$ induces an isomorphism on $H^0$ and a surjection on $H^{-1}$.  This is a local condition, so we assume that $X = \Spec A$ is affine and replace the complex of sheaves $\bL_{X/Y}$ with a complex of $A$-modules that is quasi-isomorphic to $\bL_{X/Y}$ and do the same for $\bE_\bullet$.  The condition we wish to verify is equivalent to the condition that for every $A$-modules $J$, the map
\begin{equation} \label{eqn:27}
\Ext^p(\bL_{X/Y}, J) \rightarrow \Ext^p(\bE_\bullet, J)
\end{equation}
be an isomorphism for $p = 0$ and an injection for $p = 1$ (by, for example, \cite[Lemma~A.1.1]{ACW}).  But we may identify $\Ext^0(\bL_{X/Y}, J)$ (resp.\ $\Ext^0(\bE_\bullet, J)$) as the automorphism group of the identity in $\sN_{X/Y}(A, J)$ (resp.\ $\sE(A,J)$) and $\Ext^1(\bL_{X/Y}, J)$ (resp.\ $\Ext^1(\bE_\bullet, J)$) as the group of isomorphism classes in $\sN_{X/Y}(A,J)$ (resp.\ $\sE(A,J)$).  The full faithfulness of $\sN_{X/Y} \rightarrow \sE$ implies the required facts about the maps~\eqref{eqn:27}.
\end{proof}

\subsubsection*{Virtual fundamental classes} \label{sec:vfc}

Let $\sE$ be an additively cofibered, left exact, cartesian category over $X$.  For each $X$-algebra $A$, let $M_A$ be the $A$-module $A$ itself.  The collection of pairs $(A, M_A)$ constitutes a  cocartesian section of $X\hAlgMod$ over $X\hAlg$, and therefore by pullback via this section we obtain a stack $\fE$ on $X$.  This is an $A$-module stack (the analogue of an abelian group stack, or ``Picard stack,'' for $A$-modules).  In particular, it has an action of $\bA^1$, in the sense of \cite[Definition~1.5]{BF}.

If $X \rightarrow Y$ is a morphism of Deligne--Mumford type, write $\fN_{X/Y}$ for the stack obtained in this way from $\sN_{X/Y}$.

Suppose that $\sE$ is an obstruction theory for $X$ over $Y$.  Then we obtain a map $\fN_{X/Y} \rightarrow \fE$ from the map $\sN_{X/Y} \rightarrow \sE$.  This map is an embedding, by Corollary~\ref{cor:ins}.

Recall that there is a canonical embedding of the intrinsic normal cone $\fC$ in $\fN$ \cite[Definition~3.10]{BF}.  This induces an embedding $\fC \subset \fE$.  If $\fE$ is a vector bundle stack then we obtain a virtual fundamental class $[X/Y]^{\vir}_{\fE}$ by intersecting $\fC$ with the zero locus in $\fE$.

\subsection{Perfection} \label{sec:perf}

\begin{definition} \label{def:lfp}
  \begin{enumerate}
  \item We say that a left exact additively cofibered category $\sE$ over $X\hAlgMod$ is \emph{locally of finite presentation} if the natural map
  \begin{equation*}
    \varinjlim_j \sE(A_j, I_j) \rightarrow \sE(A, I)
  \end{equation*}
  is an equivalence whenever $\set{(A_j, I_j)}_j$ is a filtered system of objects of $X\hAlgMod$.
  \item A left exact additively cofibered category $\sE$ over $X\hAlgMod$ is called \emph{quasi-perfect} if the stack of $\cO_X$-modules $\fE$, whose value on an $X$-algebra $A$ is $\sE(A, A)$, is a vector bundle stack \cite[Definition~1.9]{BF} over $X$.

  \item A left exact additively cofibered category $\sE$ is called \emph{perfect} if it is locally of finite presentation and quasi-perfect.
  \end{enumerate}
\end{definition}

\begin{lemma}
Suppose that $\fE$ is a vector bundle stack on a Deligne--Mumford stack $X$.  Then there is a morphism of sheaves of flat $\cO_X$-modules $\bF_1 \rightarrow \bF_0$ such that $\fE = [\bF_0 / \bF_1]$.
\end{lemma}
\begin{proof}
By definition, there is, \'etale-locally in $X$, a surjective linear map from a vector bundle onto $\fE$.  Choose an \'etale cover of $X$ such that on each $U$ in the cover there is a surjection $V_U \rightarrow \fE_U$.  Take $\bF_0$ to be the direct sum $\sum_{i : U \rightarrow X} i_! V_U$ and let $\bF_1 = \bF_0 \fp_{\fE} e$ be the kernel of the map $\bF_0 \rightarrow \fE$.  Then $\bF_0$ and $\bF_1$ are sheaves of $\cO_X$-modules and $\fE = [\bF_0 / \bF_1]$ since $\bF_0$ surjects onto $\fE$.

To check that $\bF_0$ is flat, it suffices to check that each $i_! V$ is flat if $V$ is.  We need to see that $W \mapsto i_! V \tensor W$ is left exact.  But $i_! V \tensor W = i_!(V \tensor i^\ast W)$ and $W \mapsto i_!(V \tensor i^\ast W)$ is the composition of the exact functors $i^\ast$, $V \tensor (-)$, and $i_!$.

We check that $\bF_1$ is also flat by showing that $\uTor_1(\bF_1, J) = 0$ for all sheaves of $\cO_X$-modules $J$.  The complex $\bF_\bullet$ has perfect amplitude in $[-1,0]$, so $\uTor_i(\bF_\bullet, -)$ vanishes for $i \not\in [0,1]$.  On the other hand, we have an exact sequence
\begin{equation*}
\uTor_2(\bF_\bullet, J) \rightarrow \uTor_1(\bF_1, J) \rightarrow \uTor_1(\bF_0, J)
\end{equation*}
for any sheaf of $\cO_X$-modules $J$.  As just noted, the first term of this sequence vanishes because $\bF_\bullet$ has perfect amplitude in $[-1,0]$ and the last term vanishes because $\bF_0$ was demonstrated above to be flat.
\end{proof}

\begin{proposition} \label{prop:vb}
If $\fE$ is a vector bundle stack on $X$ then there is a complex $\bE_\bullet$ of flat $\cO_X$-modules in degrees $[-1,0]$ and equivalences $\fE(A) \simeq \Psi_{\bE_\bullet}(A)$ for each $X$-algebra $A$, compatible with the covariance of both terms with respect to $A$.
\end{proposition}
\begin{proof}
By Corollary~\ref{cor:vbdual}, $\fE^\vee$ is also a vector bundle stack.  Therefore by the lemma there is a $2$-term complex of flat $\cO_X$-modules $\bE_\bullet$ in degrees $[-1,0]$ such that $\fE^\vee = [\bE_0 / \bE_1]$ and we have $(\fE^\vee)^\vee \simeq \fE$ (Corollary~\ref{cor:vbdd}).  On the other hand, a section of $(\fE^\vee)^\vee$ over $\Spec A$, for $A$ an $X$-algebra, can be identified with a morphism of $\tA$-module stacks $\fE^\vee \rightarrow B \tA$.  By left exactness, that is the same as a morphism $\bE_0 \tensor_{\cO_X} \tA \rightarrow B \tA$ and an isomorphism between the induced map $\bE_1 \tensor_{\cO_X} \tA \rightarrow B \tA$ and the zero map.  In other words, it is a commutative diagram
\begin{equation*} \xymatrix{
& & & \bE_1 \tensor_{\cO_X} \tA\ar[d] \ar[dl] \\
0 \ar[r] & \tA \ar[r] & Z \ar[r] & \bE_0 \tensor_{\cO_X} \tA \ar[r] & 0
} \end{equation*}
in which the second row is exact (cf.\ the proof of Lemma~\ref{lem:vbcx} in the appendix).  But to give such a diagram is the same as to give an object of $\Psi_{\bE_\bullet}(A, A)$, by definition.
\end{proof}

\begin{proposition}\label{prop:perf}
Suppose that $X$ is a Deligne--Mumford stack and $\sE$ is a cartesian, left exact, additively cofibered category on $X\hAlgMod$ and that the associated abelian cone stack $\fE$ of $\sE$ is a vector bundle stack.  Assume also that $\sE$ is locally of finite presentation and $X$ is locally of finite presentation over $Y$.  Then there is a perfect complex $\bE_\bullet$, of perfect amplitude in $[-1,0]$, and an equivalence of fibered categories
\begin{equation*}
\sE \simeq \Psi_{\bE_\bullet}
\end{equation*}
that is determined, up to unique isomorphism, by $\fE$.
\end{proposition}
\begin{proof}
The first step will be to construct the map $\Psi_{\bE_\bullet} \rightarrow \sE$.  Then we will show it is an equivalence.

By Proposition~\ref{prop:vb}, for $\fE$ to be a vector bundle stack means that the cofibered category obtained by restricting $\sE$ to pairs $(A, I)$ where $I$ is isomorphic to $A$ is equivalent to the cofibered category obtained by restricting $\Psi_{\bE_\bullet}$ to the same subcategory, for some chain complex $\bE_\bullet$ of $\cO_X$-modules of perfect amplitude in $[-1,0]$.  Since both $\Psi_{\bE_\bullet}$ and $\sE$ are additively cofibered, this equivalence can be extended in an essentially unique way to the category of pairs $(A,I)$ where $I$ is a free $A$-module.  Since both $\Psi_{\bE_\bullet}$ and $\sE$ are stacks over $X\hAlg$, this equivalence can be further extended in an essentially unqiue way to the category of $(A,I)$ where $I$ is a \emph{locally free} $A$-module.  Since both are left exact, this equivalence extends even to the cofibered categories $\Psi_{\bE_\bullet}$ and $\sE$ on pairs $(A,\bI^\bullet)$ where $\bI^\bullet$ is a cochain complex of locally free $A$-modules.

The proposition asserts that the isomorphism between $\sE$ and $\Psi_{\bE_\bullet}$ is unique, which means that if the proposition is proved locally, the local statements will automatically glue to a global statement.  That is, to construct the maps $\Psi_{\bE_\bullet}(A,I) \xrightarrow{\sim} \sE(A, I)$ is a local problem on the \'etale site of $A$, provided that we also demonstrate they are uniquely determined and compatible with further localization.  We can therefore work locally and assume (using \cite[Corollary~4.3]{perf}) that there is a quasi-isomorphism $\tbF_\bullet \rightarrow \bE_\bullet$ where $\bF_\bullet$ is a $2$-term complex of \emph{free} $A$-modules, with $F_i = 0$ for $i \not= 0,1$, and $\tbF_\bullet$ is its associated complex of sheaves.  The perfect complex $\bF_\bullet$ is not unique, however, so we will to argue later that our constructions depend on $\bF_\bullet$ only up to unique isomorphism.  

Now that we are working locally, say over a ring $A$, we can use the fact that $\sE$ and $\Psi_{\bE_\bullet}$ are cartesian to reduce the problem to constructing an equivalence between their restrictions to the category of $A$-modules.  Indeed, if $\varphi : A \rightarrow B$ is a homomorphism of $X$-algebras  and $J$ is a $B$-module, then $\Ext_B(\bE_\bullet(B), J) \rightarrow \Ext_A(\bE_\bullet(A), J_{[\varphi]})$ is an equivalence, as is $\sE(B,J) \rightarrow \sE(A,J_{[\varphi]})$.  Accordingly, we now drop reference to the ring $A$ from the notation, and simply write $\Psi_{\bE_\bullet}(I)$, $\Ext(\bE_\bullet, \tI)$ and $\sE(I)$ in place of $\Psi_{\bE_\bullet}(A,I)$, $\Ext_A(\bE_\bullet(A), \tI)$ and $\sE(A,I)$ for an $A$-module $I$.

There is an object $\xi \in \Psi_{\bE_\bullet}(\bF_\bullet) = \Ext(\bE_\bullet, \tbF_\bullet[-1])$, well-defined up to unique isomorphism, corresponding to the identity map $\id_{\bF_\bullet}$ via the equivalence (\ref{lem:psi-equiv}):
\begin{equation*}
\Psi_{\bE_\bullet}(\bF_\bullet[-1]) \rightarrow \Psi_{\bF_\bullet}(\bF_\bullet[-1]) .
\end{equation*}
Let $\eta$ be the image of $\xi$ via the equivalence
\begin{equation*}
\Psi_{\bE_\bullet}(\bF_\bullet[-1]) \rightarrow \sE(\bF_\bullet[-1]).
\end{equation*}
By~\ref{lem:adj}, this extends, uniquely up to unique isomorphism, to a cocartesian section of $\sE$ over $\Psi_{\bF_\bullet}$.

We have therefore constructed a morphism $\Psi_{\bF_\bullet} \rightarrow \sE$.  By composition with the equivalence $\Psi_{\bE_\bullet} \rightarrow \Psi_{\bF_\bullet}$, we obtain a map
\begin{equation}\label{eqn:7}
\Psi_{\bE_\bullet} \rightarrow \sE
\end{equation}
again determined uniquely up to unique isomorphism from $\bF_\bullet \rightarrow \bE_\bullet$.

In the above discussion, every choice was determined uniquely up to unique isomorphism except for the map $\tbF_\bullet \rightarrow \bE_\bullet$.  In fact, this map is uniquely determined up to contractible ambiguity in a $2$-category of $2$-term chain complexes, but instead of using this we will verify explicitly that the definition of the map $\Psi_{\bE_\bullet} \rightarrow \sE$ depends on the choice of the quasi-isomorphism $\tbF_\bullet \rightarrow \bE_\bullet$ only up to unique isomorphism.

Suppose that we are given two quasi-isomorphisms $u : \tbF_\bullet \rightarrow \bE_\bullet$ and $v : \tbF'_\bullet \rightarrow \bE_\bullet$ where $\bF_\bullet$ and $\bF'_\bullet$ are chain complexes of free $A$-modules, concentrated in degrees $[-1,0]$.   Then by \cite[4.2]{perf}, there is a quasi-isomorphism $w : \bF_\bullet \rightarrow \bF'_\bullet$ and a chain homomotpy $h$ between $u$ and the induced map $v \tw : \tbF_\bullet \rightarrow \bE_\bullet$ (recall that $\tw : \tbF_\bullet \rightarrow \tbF'_\bullet$ is the morphism of complexes of sheaves associated to $w$).  Suppose that $w' : \bF_\bullet \rightarrow \bF'_\bullet$ is another morphism and that $h'$ is a chain homotopy connecting $v\tw'$ and $u$.  Then $h - h'$ is a chain homotopy between $v(\tw - \tw')$ and the zero map.

\begin{lemma}
Suppose that $w : \bF_\bullet \rightarrow \bF'_\bullet$ is a morphism of $2$-term complexes of free $A$-modules and $v : \tbF'_\bullet \rightarrow \bE_\bullet$ is a quasi-isomorphism of $\tA$-modules on $\et(A)$.  Then any chain homotopy $v\tw \simeq 0$ is induced from a \emph{unique} chain homotopy $w \simeq 0$.
\end{lemma}
\begin{proof}
First we note that chain homotopies between $w$ and zero form a pseudo-torsor under $\Hom(\bF_\bullet, \bF'_\bullet[-1]) = \Hom(H_0(\bF_\bullet), H_1(\bF'_\bullet))$.  Analogously, chain homotopies between $v \tw$ and zero form a torsor under $\Hom(\tbF_\bullet, \bE_\bullet[-1]) = \Hom(H_0(\bF_\bullet)^\sim, H_1(\bE_\bullet))$.  The map
\begin{equation*}
\Hom(H_0(\bF_\bullet), H_1(\bF'_\bullet)) \rightarrow \Hom(H_0(\bF_\bullet)^\sim, H_1(\bE_\bullet))
\end{equation*}
is a bijection because $v$ is a quasi-isomorphism.  Therefore if $h : vw \simeq 0$ is a chain homotopy in $\Hom(\bF_\bullet, \bE_\bullet)$, there is at most one lift of $h$ to a chain homotopy between $v$ and $0$ in $\Hom(\bF_\bullet, \bF'_\bullet)$.

Now we argue that a chain homotopy lifting $h$ must exist.  Since $v$ is a quasi-isomorphism and $vw$ induces the zero map on homology, $w$ also induces the zero map on homology.  We therefore have a commutative diagram with exact rows:
\begin{equation*} \xymatrix{
0 \ar[r] & H_1(\bF_\bullet) \ar[r] \ar[d]_0 & \bF_1 \ar[r]^d \ar[d]_{w_1} & \bF_0 \ar@{-->}[dl]_{g} \ar[d]^{w_0} \ar[r] \ar[dr]^0 & H_0(\bF_\bullet) \ar[r] \ar[d]^0 & 0 \\
0 \ar[r] & H_1(\bF'_\bullet) \ar[r] & \bF'_1 \ar[r]_d & \bF'_0 \ar[r] & H_0(\bF'_\bullet) \ar[r] & 0
} \end{equation*}
Since $\bF_0$ is free, there is certainly a dashed arrow $g$ such that $d g = w_0$.  The failure of this map to be a chain homotopy is measured be the difference $g d - w_1$.  However, $d  g d - dw_1 = w_0 d - w_0 d = 0$, so the map $g d - w_1$ factors through a map $\bF_1 \rightarrow H_1(\bF'_\bullet)$.  Moreover, this map vanishes on $H_1(\bF_\bullet) \subset \bF_1$, so it factors through a map $\bF_1 / H_1(\bF_\bullet) \rightarrow H_1(\bF'_\bullet)$.  We would like to extend this map to a map $\bF_0 \rightarrow H_1(\bF'_\bullet)$, for then we could subtract this map from $g$ to get the chain homotopy we are looking for.  The obstruction to the existence of such a lift lies in $\Ext^1(H_0(\bF_\bullet), H_1(\bF'_\bullet))$ and does not depend on the choice of the original map $g$.  But the map
\begin{equation*}
\Ext^1(H_0(\bF_\bullet), H_1(\bF'_\bullet)) \rightarrow \Ext^1(H_0(\tbF_\bullet), H_1(\bE_\bullet))
\end{equation*}
is an isomorphism (because $v$ is a quasi-isomorphism) and by construction, it sends the obstruction to the existence of a chain homotopy $w \simeq 0$ to the obstruction to the existence of a chain homotopy $vw \simeq 0$.  The latter obstruction is zero by assumption, and therefore so is the former.
\end{proof}

By the lemma, there is a \emph{unique} chain homotopy $w \simeq w' : \bF_\bullet \rightarrow \bF'_\bullet$.  This induces a uniquely determined isomorphism $w_\ast \xi \simeq w'_\ast \xi \in \Psi_{\bE_\bullet}(\bF'_\bullet[-1])$ and therefore a uniquely determined isomorphism between the induced objects of $\sE(\bF'_\bullet)$.  This proceeds to give an isomorphism between the induced cocartesian sections of $\sE$ over $\Psi_{\bF'_\bullet}$, and therefore between the induced maps $\Psi_{\bE_\bullet} \rightarrow \sE$.

\vskip .5cm
For the rest of this proof, we write $\sF$ for the additively cofibered stack $\Psi_{\bE_\bullet}$.  To demonstrate that the map $\sF \rightarrow \sE$ is an equivalence is a local problem, so we are free to assume that $X$ is representable by a commutative ring $A$ and $\bE_\bullet$ is representable by a $2$-term complex of free $A$-modules.

We note that because $X$ is assumed to be locally of finite presentation, any $X$-algebra $A$ admits a map from an $X$-algebra $A_0$ whose underlying commutative ring is finitely generated over $\bZ$.  Likewise, the local finite presentation of $\sE$ implies that any object of $\sE(A, I)$ is induced from some object of $\sE(A_1, I_1)$ where $A_1$ is an $A_0$-algebra of finite type and $I_1$ is an $A_1$-module of finite type.  Moreover, by the same argument, any two such representations of an object of $\sE(A,I)$ as objects over finite type objects of $\ComRngMod$ can be compared over a finite type object.  It follows that to demonstrate $\sF \rightarrow \sE$ is an equivalence, it suffices to check that $\sF(A, I) \rightarrow \sE(A,I)$ is an equivalence whenever $A$ is an $X$-algebra whose underlying commutative ring is of finite type over $\bZ$ and $I$ is an $A$-module of finite type.

In particular, we can assume that $A$ is noetherian and therefore that $I$ admits a filtration of finite length whose associated graded module is a direct sum of modules $A/\frkp$ where $\frkp$ is a prime ideal of $A$ \cite[Theorem~6.4]{M2}.  We note furthermore that $\sF$ is an \emph{exact} additively cofibered category.  Applying the following lemma inductively completes the proof.
\begin{lemma}
  Suppose that $\sF \rightarrow \sE$ is a morphism of left exact additively cofibered categories over an abelian category $\sC$ and that
  \begin{equation*}
    0 \rightarrow I \rightarrow J \rightarrow K \rightarrow 0
  \end{equation*}
  is an exact sequence in $\sC$ such that $\sF(I) \rightarrow \sE(I)$ and $\sF(K) \rightarrow \sE(K)$ are isomorphisms.  Assume also that $\sF$ is right exact.  Then $\sF(J) \rightarrow \sE(J)$ is also an isomorphism.
\end{lemma}
  We have to show that $\sF(J) \rightarrow \sE(J)$ is fully faithful and essentially surjective.  For an additively cofibered category $\sG$, write $\sG_1(L)$ and $\sG_0(L)$, respectively, for the isomorphism classes in $\sG(L)$ and the automorphism group of the identity object of $\sG(L)$.  We use the exact sequence \cite[(2.5.2)]{ccct} associated to $\sF$ and $\sE$ (note, however, that our notation differs:  in loc.\ cit.\ $\sF^0$ denotes the isomorphism group of the identity and $\sF^1$ denotes the set of isomorphism classes).  The rows of the commutative diagram
  \begin{equation*} \xymatrix@R=10pt{
    0 \ar[r] & \sF_1(I) \ar[r] \ar[d] & \sF_1(J) \ar[r] \ar[d] & \sF_1(K) \ar[r] \ar[d] & \sF_0(I) \ar[d] \\
    0 \ar[r] & \sE_1(I) \ar[r] & \sE_1(J) \ar[r] & \sE_1(K) \ar[r] & \sE_0(I)
  } \end{equation*}
  are exact, so the $5$-lemma implies the faithfulness; the exactness of the rows of
  \begin{equation*} \xymatrix@R=10pt{
    \sF_1(K) \ar[r] \ar[d] & \sF_0(I) \ar[r] \ar[d] & \sF_0(J) \ar[r] \ar[d] & \sF_0(K) \ar[d] \ar[r] & 0 \\
    \sE_1(K) \ar[r]  & \sE_0(I) \ar[r]  & \sE_0(J) \ar[r]  & \sE_0(K)  
  } \end{equation*}
  and a version of the $5$-lemma implies that the functor is full and essentially surjective.  Note that the surjectivity of $\sF_0(J) \rightarrow \sF_0(K)$ comes from the exactness of $\sF$.
\end{proof}

\section{Obstruction groups} \label{sec:obgrp}

Let $\sE$ be an obstruction theory for a morphism of homogeneous stacks $h : X \rightarrow Y$.  For each $X$-algebra $A$ and each $A$-module $I$, write $\sE_0(A, I)$ for the abelian group of isomorphism classes in $\sE(A,I)$ and $\sE_1(A,I)$ for the abelian group of automorphisms of the identity section of $\sE(A,I)$ (cf.\ \cite[Section~1.6]{ccct}, but note the difference in the indexing).

\begin{lemma}
Suppose that $A_0$ is an $A$-algebra such that the $A$-module $I$ is induced from an $A_0$-module $I_0$.  Then the $A$-module structures on $\sE_0(A,I)$ and $\sE_1(A,I)$ is induced from the $A_0$-module structures on $\sE_0(A_0,I_0)$ and $\sE_1(A_0,I_0)$.
\end{lemma}
\begin{proof}
One may use the fact that $\sE_i(A,I) \simeq \sE_i(A_0,I_0)$ because $\sE(A,I) \simeq \sE(A_0,I_0)$ (see Proposition~\ref{prop:cart}~\ref{cart:3}).

Alternatively, note that the $A$-module structure on $\sE_i(A,I)$ is induced from the map $A \rightarrow \End_A(I) \rightarrow \End(\sE_i(A,I))$ by the functoriality of $\sE_i(A,I)$ with respect to $A$-module homomorphisms.  The map $A \rightarrow \End_A(I)$ factors through $A_0$ by hypothesis.
\end{proof}

Suppose that $A$ is an $X$-algebra and $A'$ is a square-zero $Y$-algebra extension of $A$ with ideal $I$.  The corresponding object $\xi$ of $h^\ast \Exal_Y$ induces an obstruction $\omega \in \sE(A,I)$, whose isomorphism class in $\sE_0(A,I)$ we denote by $\oomega$.  Then $\oomega$ is zero if and only if $\xi$ is induced from some object of $\Exal_X$, meaning that there is an $X$-algebra structure on $A'$ inducing both the given $X$-algebra structure on $A$ and the $Y$-algebra structure on $A'$.  In geometric terms, $\oomega$ functions as an obstruction to the existence of a solution to the lifting problem
\begin{equation*} \xymatrix{
\Spec A \ar[r] \ar[d] & X \ar[d] \\
\Spec A' \ar[r] \ar@{-->}[ur] & Y .
} \end{equation*}
This implies that the obstruction groups $\sE_0(A,I)$ satisfy \cite[(2.6)~(ii)]{versal}.  The obstruction groups $\sE_0(A,I)$ are also functorial in $(A,I)$ ``in the obvious sense'' (loc.\ cit.) but they need not take finite modules to finite modules (as in \cite[(2.6)~(i)]{versal}) without an extra hypothesis.

Should $\oomega$ vanish, the lifts of $\xi$ to $\Exal_X$ are in bijection with the isomorphisms between $\omega$ and a (fixed) zero section of $\sE(A,I)$.  Such isomorphisms form a torsor under the automorphism group of the zero section, namely $\sE_1(A,I)$.  

\subsection{Obstruction sheaves}

Suppose that $X \rightarrow Y$ is a morphism of homogeneous stacks and $\sE$ is an obstrution theory for $X$ over $Y$.  For each $X$-algebra $A$ and each $A$-module $I$, define $\usE(A,I)$ to be the fibered category on $\et(A)$ whose value on an \'etale $A$-algebra $B$ is $\usE(B, I \tensor_A B)$.  Of course, $\usE(A,I)$ is a stack on $\et(A)$ because it is the restriction of a stack from the large \'etale site.

Define $\usE_1(A,I)$ and $\usE_0(A,I)$ analogously.

\begin{proposition} \label{prop:qcoh}
Suppose $X$ is of Deligne--Mumford type over $Y$, let $A$ be an $X$-algebra, and let $I$ be an $A$-module.  Then the presheaf $\usE_1(A,I)$ is a quasi-coherent sheaf.
\end{proposition}
\begin{proof}
It is clear that $\sE_1$ is a sheaf since it is the presheaf of automorphisms of a section of a stack.  It is therefore a local problem in $A$ to demonstrate that $\sE_1(A,I)$ is quasi-coherent.

By definition, we have 
\begin{equation*}
\sE_1(A,I) = e(A,I) \fp_{\sE(A,I)} e(A,I) = e(A,I) \fp_{\Exal_Y(A,I)} \Exal_X(A,I)
\end{equation*}
where $e$ is a zero section of $\sE$.  But $e(A,I) \rightarrow \Exal_Y(A,I)$ is representable by the trivial $Y$-algebra extension of $A$ by $I$.  Therefore $\sE_1(A,I)$ consists of the $X$-algebra structures on $A + I$ lifting the trivial $Y$-algebra structure.  In geometric terms, it is the set of lifts of
\begin{equation*} \xymatrix{
S \ar[r] \ar[d] & X \ar[d] \\
S[\tI] \ar@{-->}[ur] \ar[r] & Y
} \end{equation*}
where $S = \Spec A$ and $S[\tI] \rightarrow Y$ is the map induced from the map $S \rightarrow Y$ and the retraction $S[\tI] \rightarrow S$.

To prove the proposition, we must show that $\sE_1(A, I) \tensor_A B \simeq \sE_1(B, I \tensor_A B)$ whenever $B$ is an \'etale $A$-algebra.

By base change, we can assume that $Y$ is representable by a commutative ring $C$ and therefore that $X$ is a Deligne--Mumford stack.  If $U$ is \'etale over $X$ and the $X$-algebra structure on $A$ is induced from a $U$-algebra structure on $A$, then $\Exal_X(A,I) = \Exal_U(A,I)$.  Since we are free to work locally in $A$, we can therefore replace $X$ by $U$ and assume that $X$ is representable by a $C$-algebra $D$.  Let $\varphi : D \rightarrow A$ denote the map giving the $X$-algebra structure on $A$ and let $\psi : D \rightarrow B$ be the induced map.

Now we have
\begin{equation*}
\sE_1(B, I \tensor_A B) = \Hom_D(\Omega_{D/C}, (I \tensor_A B)_{[\psi]}) .
\end{equation*}
Since $B$ is flat over $A$, we have
\begin{equation*}
\Hom_D(\Omega_{D/C}, (I \tensor_A B)_{[\psi]}) = \Hom_D(\Omega_{D/C}, I_{[\varphi]}) \tensor_A B = \sE_1(A, I) \tensor_A B
\end{equation*}
which implies the proposition.
\end{proof}

Although $\usE_1(A,I)$ is always a sheaf, $\usE_0(A,I)$ generally will not be.  However, we have
\begin{proposition}
The functor $B \mapsto \sE_0(B, I \tensor_A B)$ is a separated presheaf on the \'etale site of $A$.
\end{proposition}
\begin{proof}
In effect, we must show that if $\omega$ is an object of $\sE(A, I)$ that is locally isomorphic to zero then it is globally isomorphic to zero.  But if $\omega$ is locally isomorphic to zero, the isomorphisms between $\omega$ and zero form a torsor on the \'etale site of $A$ under the sheaf of abelian groups $\usE_1(A,I)$.  By Proposition~\ref{prop:qcoh}, $\usE_1(A,I)$ is quasi-coherent and therefore a torsor under $\usE_1(A,I)$ on $\et(A)$ admits a global section.
\end{proof}

\begin{corollary}
\begin{enumerate}[label=(\roman{*})]
\item The map from $\usE_0(A,I)$ to its associated sheaf is injective.
\item If $\omega \in \usE_0(A,I)$ is an obstruction to a lifting problem then its image in the associated sheaf of $\usE_0(A,I)$ is also an obstruction to the same lifting problem.
\end{enumerate}
\end{corollary}

\subsection{Li--Tian obstruction theories}

Let $M_A$ denote the $A$-module structure on $A$ itself.  For each $A$-module $I$, there is a natural map of presheaves of $\tA$-modules on $\et(A)$,
\begin{equation} \label{eqn:17}
\usE_0(A, M_A) \tensor_A I \rightarrow \usE_0(A, I) ,
\end{equation}
induced by the $A$-linear map $A \rightarrow I$ associated to each element of $I$.  The linearity with respect to $I$ follows from the additivity of $\sE_0(A,-)$.  The functor $I \mapsto \sE_0(A,M_A) \tensor_A I$ is right exact, so if~\eqref{eqn:17} is an isomorphism for all $I$ then $\sE_0(A, -)$ is right exact.  
\begin{proposition}
If the maps~\eqref{eqn:17} are isomorphisms for all $I$ then $\sE$ is an \emph{exact} additively cofibered category.  If $\sE$ is exact and locally of finite presentation then the maps~\eqref{eqn:17} are isomorphisms.
\end{proposition}
\begin{proof}
If $J \rightarrow K$ is a surjection of $A$-modules, then right exactness of $\usE_0(A,-)$ implies that $\usE(A,J) \rightarrow \usE(A,K)$ is locally surjective on isomorphism classes, which is what it means to be right exact.  This proves the first claim.

Conversely, suppose that $\sE$ is exact, so $\sE_0$ is right exact.  Then for any exact sequence
\begin{equation} \label{eqn:18} 
0 \rightarrow I \rightarrow J \rightarrow K \rightarrow 0
\end{equation}
of $A$-modules, we have a commutative diagram
\begin{equation*} \xymatrix@R=12pt{
\usE_0(A,M_A) \tensor_A I \ar[r] \ar[d] & \usE_0(A,M_A) \tensor_A J \ar[r] \ar[d] & \usE_0(A,M_A) \tensor_A K \ar[r] \ar[d] & 0 \\
\usE_0(A,I) \ar[r] & \usE_0(A,J) \ar[r] & \usE_0(A,K) \ar[r] & 0 
} \end{equation*}
with exact rows.  By additivity and local finite presentation, the map $\sE_0(A,M_A) \tensor_A J \rightarrow \sE_0(A,J)$ is an isomorphism if $J$ is a free $A$-module.  If $K$ is given, we may choose an exact sequence~\eqref{eqn:18} in which $J$ is free.  Then the commutative diagram above implies the surjectivity of $\sE_0(A,M_A) \tensor_A K \rightarrow \sE_0(A,K)$.  This applied with $K$ replaced by $I$ gives the surjectivity of the leftmost vertical arrow, which is enough to imply that the map $\sE_0(A,M_A) \tensor_A K \rightarrow \sE_0(A,K)$ is a bijection, by the $5$-lemma.
\end{proof}

Let $\usE_0(A,I)^+$ denote the associated sheaf of $\usE_0(A,I)$ on $\et(A)$.

\begin{proposition}
If $\sE$ is an exact obstruction theory for $X$ over $Y$ then it induces an obstruction theory for $X$ over $Y$ in the sense of \cite[Definition~1.2]{LT}, such that sheaf of obstruction groups is $\usE_1(A,I)^+$.
\end{proposition}
\begin{proof} 
We recall the situation of \cite[Definition~1.2]{LT} (but change some of the letters).  Suppose that $S_0 \rightarrow S \rightarrow S'$ is a sequence of closed embeddings of schemes with $S_0 = \Spec A_0$ affine.  Assume that the ideal of $S_0$ in $S$ annihilates the ideal of $S$ in $S'$ and write $I$ for the $A_0$-module corresponding to the ideal of $S$ in $S'$.  Given a lifting problem
\begin{equation*} \xymatrix{
S \ar[r] \ar[d] & X \ar[d] \\
S' \ar@{-->}[ur] \ar[r] & Y
} \end{equation*}
there is an obstruction to the existence of a lift in $\sE_0(A_0, I)$.  Since $\sE_0(A_0, I) \subset \Gamma(\et(A), \usE_0(A,I)^+)$ we can also view the obstruction as living in $\usE_0(A,I)^+$.  Since $\sE$ is exact and locally of finite presentation, $\sE_0(A_0, I) = \sE_0(A_0, M_{A_0}) \tensor_{A_0} I$.  Since sheafification preserves tensor product, we also get 
\begin{equation*}
\usE_0(A_0, I)^+ = \usE_0(A_0, M_{A_0})^+ \tensor_{A_0} I .
\end{equation*}
We determine that there is an obstruction class in $\usE_0(A_0, M_{A_0})^+ \tensor_{A_0} I$, as required by~\cite[Definition~1.2]{LT}.
\end{proof}

Note that an obstruction theory that is represented by a $2$-term complex of free $A$-modules is exact.  Therefore any perfect obstruction theory, in the sense of \cite{BF} induces an obstruction theory in the sense of \cite{LT} (which is perfect in the sense of Li--Tian).  Moreover, this is the same process by which a Li--Tian obstruction theory is associated to a Behrend--Fantechi obstruction theory in \cite[Section~3]{KKP}.  Provided that the complex representing the obstruction theory is globally representable by a $2$-term complex of vector bundles (in order that the Li--Tian obstruction class be defined), \cite[Corollary~1]{KKP} implies that the virtual fundamental class defined by Li and Tian coincides with that defined by Behrend and Fantechi (as in Section~\ref{sec:vfc}).

\section{Compatibility} \label{sec:compat}

\subsection{Compatible obstruction theories for morphisms of Deligne--Mumford type} \label{sec:compat-dm}

We recall that an exact sequence of left exact, additively cofibered categories
\begin{equation*}
0 \rightarrow \sE \rightarrow \sF \rightarrow \sG
\end{equation*}
is a specified cartesian diagram
\begin{equation*} \xymatrix{
\sE \ar[r] \ar[d] & \sF \ar[d] \\
e \ar[r] &  \sG .
} \end{equation*}
We say that the sequence is right exact, and we affix an arrow $\sG \rightarrow 0$ at the end of the sequence, to mean that the map $\sF \rightarrow \sG$ is locally essentially surjective.

Let $X \xrightarrow{f} Y \xrightarrow{g} Z$ be a sequence of morphisms of homogeneous stacks.  We will specify a notion of compatibility between relative obstruction theories for $f$, $g$, and $gf$ that reduces to a familiar notion for obstruction theories in the sense of Behrend--Fantechi.  It is easiest to describe compatibility for morphisms of Deligne--Mumford type, so we do that first; then we explain how this can be generalized.

\begin{definition} \label{def:compat}
Let $X \xrightarrow{f} Y \xrightarrow{g} Z$ be a sequence of morphisms of Deligne--Mumford type.  A compatible sequence of relative obstruction theories for this sequence of maps is a commutative diagram
\begin{equation} \label{eqn:11}\xymatrix{
0 \ar[r] & \sN_{X/Y} \ar[r] \ar[d] & \sN_{X/Z} \ar[r] \ar[d] & f^\ast \sN_{Y/Z} \ar[d] \\
0 \ar[r] & \sE_{X/Y} \ar[r] & \sE_{X/Z} \ar[r] & f^\ast \sE_{Y/Z} \ar[r] & 0
} \end{equation}
in which the first row is the canonical sequence of relative instrinsic normal stacks (whose left exactness we leave to the reader),  the second row is an exact sequence of left exact additively cofibered stacks, and the map $f^\ast \sN_{Y/Z} \rightarrow f^\ast \sE_{Y/Z}$ is pulled back from an obstruction theory $\sN_{Y/Z} \rightarrow \sE_{Y/Z}$ for $Y$ over $Z$.  The map $\sN_{X/Y} \rightarrow \sE_{X/Y}$ is required to be isomorphic to the one induced from the commutative square on the right by the exactness of the two rows (provided it exists the isomorphism is unique).
\end{definition}

\begin{proposition}
Suppose we have a compatible sequence of obstruction theories as in Definition~\ref{def:compat} and the obstruction theories in question are all \emph{perfect}, hence give rise to relative Gysin pullback functors.  Then
\begin{gather*}
(gf)^! = f^! g^! \qquad \text{and} \\
f^! [Y/Z]^{\vir} = [X/Z]^{\vir} .
\end{gather*}
\end{proposition}
\begin{proof}
This reduces immediately to \cite[Theorem~4]{Man}.
\end{proof}

\subsection{Compatible obstruction theories in general}

Suppose that $X \xrightarrow{f} Y \xrightarrow{g} Z$ is a sequence of morphisms of homogeneous stacks.  We say that relative obstruction theories $\sE_{X/Y}$, $\sE_{X/Z}$, and $\sE_{Y/Z}$ are \emph{compatible} if we are given a commutative diagram
\begin{equation} \label{eqn:22} \xymatrix{
\Exal_X \ar[r] \ar[d] & f^\ast \Exal_Y \ar[r] \ar[d] & f^\ast g^\ast \Exal_Z \ar[d] \\
e \ar[r] & \sE_{X/Y} \ar[r] \ar[d] & \sE_{X/Z} \ar[d] \\
& e \ar[r] & f^\ast \sE_{Y/Z} .
} \end{equation}
in which
\begin{enumerate}
\item the square in the upper left is the cartesian square associated to the obstruction theory $\sE_{X/Y}$,
\item the upper large rectangle is the cartesian square associated to the obstruction theory $\sE_{X/Z}$,
\item the large rectange on the right is the cartesian square associated to the obstruction theory $\sE_{Y/Z}$, pulled back via $f^\ast$,
\item the square in the lower right is cartesian, and
\item the map $\sE_{X/Z} \rightarrow f^\ast \sE_{Y/Z}$ is locally essentially surjective.
\end{enumerate}

\begin{proposition}
Suppose that $X \rightarrow Y \rightarrow Z$ are morphisms of Deligne--Mumford type.  If $\sE_{X/Y}$, $\sE_{X/Z}$, and $\sE_{Y/Z}$ are compatible relative obstruction theories in the sense described above then they are also compatible relative obstruction theories in the sense of Section~\ref{sec:compat-dm}.
\end{proposition}
\begin{proof}
The commutative diagram~\eqref{eqn:22} induces a commutative diagram
\begin{equation*} \xymatrix@R=10pt{
f^\ast \Exal_Y \ar[r] \ar[d] & f^\ast g^\ast \Exal_Z \ar[r] \ar[d] & g^\ast \Exal_Z \ar[d] \\
\sE_{X/Y} \ar[r] \ar[d] & \sE_{X/Z}\ar[d]  \ar[r] & \sE_{Y/Z} \ar[d] \\
X \ar@{=}[r] & X \ar[r] & Y .
} \end{equation*}
Recall that the restriction of $f^\ast \Exal_Y$ to the small \'etale site of $X$ is the same as the restriction of $\sN_{X/Y}$ to the small \'etale site of $X$.  Likewise $f^\ast g^\ast \Exal_Z$ and $\sN_{X/Z}$ have the same restriction to the small \'etale site of $X$ and $g^\ast \Exal_Z$ has the same restriction as $\sN_{Y/Z}$ to the small \'etale site of $Y$.  Restricting the maps above to the small sites and then pulling back to the big sites, we obtain a commutative diagram
\begin{equation*} \xymatrix@R=10pt{
\sN_{X/Y} \ar[r] \ar[d] & \sN_{X/Z} \ar[r] \ar[d] & \sN_{Y/Z} \ar[d] \\
\sE_{X/Y} \ar[r] \ar[d] & \sE_{X/Z} \ar[r] \ar[d] & \sE_{Y/Z} \ar[d] \\
X \ar@{=}[r] & X \ar[r] & Y
} \end{equation*}
which induces the required diagram~\eqref{eqn:11}.
\end{proof}

\subsection{A construction for compatible obstruction theories}

Suppose that $X \xrightarrow{f} Y \xrightarrow{g} Z$ is a sequence of morphisms of homogeneous stacks and $\sE_{X/Z}$ and $\sE_{Y/Z}$ are relative obstruction theories.  We therefore have cartesian diagrams
\begin{equation*} \xymatrix{
\Exal_X \ar[r] \ar[d] & e \ar[d] \\
f^\ast g^\ast \Exal_Z \ar[r] & \sE_{X/Z}
} \qquad \xymatrix{
f^\ast \Exal_Y \ar[r] \ar[d] & e \ar[d] \\
f^\ast g^\ast \Exal_Z \ar[r] & \sE_{Y/Z} .
} \end{equation*}
Suppose that these diagrams can be fit together in a commutative diagram
\begin{equation} \label{eqn:23} \xymatrix{
f^\ast \Exal_Y \ar[r]  \ar[dr] & f^\ast g^\ast \Exal_Z \ar[r] \ar[dr] & \sE_{X/Z} \ar[d] \\
\Exal_X \ar[u] \ar[ur] \ar[r] & e \ar[r] \ar[ur] & f^\ast \sE_{Y/Z}
} \end{equation}
over $X\hAlgMod$.  Given such a diagram, define $\sE_{X/Y}$ to be the kernel of the map $\sE_{X/Z} \rightarrow f^\ast \sE_{Y/Z}$.  That is, define $\sE_{X/Y} = \sE_{X/Z} \fp_{f^\ast \sE_{Y/Z}} e$.

\begin{lemma}
If $\sE_{X/Z}$ and $\sE_{Y/Z}$ are relative obstruction theories for $X/Z$ and $Y/Z$, respectively, fitting into a commutative diagram~\eqref{eqn:23}, and $\sE_{X/Y}$ is defined as above, then $\sE_{X/Y}$ is naturally equipped with the structure of an obstruction theory for $X$ over $Y$.
\end{lemma}
\begin{proof}
The functoriality properties of $\sE_{X/Y}$ are all deduced from those of $\sE_{X/Z}$ and $f^\ast \sE_{Y/Z}$.  The only thing to verify is that $\sE_{X/Y}$ comes equipped with a diagram~\eqref{eqn:6}.

Following the upper horizontal arrow across~\eqref{eqn:23}, we get a map $f^\ast \Exal_Y \rightarrow \sE_{X/Z}$ and an isomorphism between the induced map $f^\ast \Exal_Y \rightarrow f^\ast \sE_{Y/Z}$ and the zero map.  It therefore induces (uniquely up to unique isomorphism) a map $f^\ast \Exal_Y \rightarrow \sE_{X/Y}$.  We therefore obtain a commutative diagram~\eqref{eqn:22}.  The lower right square is cartesian by definition and the large rectangle on the right is cartesian by hypothesis; therefore the upper right square is cartesian.  The large upper rectangle is also cartesian by hypothesis, so the square in the upper left is cartesian as well, as was required.
\end{proof}


\section{Examples} \label{sec:ex}

\subsection{Smooth morphisms}

Recall the definition of $\sT_{X/Y}$ from Section~\ref{sec:tang}.  We write $B  \sT_{X/Y}$ for the additively cofibered category of $\sT_{X/Y}$-torsors:  an object of $B \sT_{X/Y}(A,I)$ is a torsor on $\et(A)$ under the sheaf of groups $\usT_{X/Y}(A,I)$.  

\begin{proposition} \label{prop:smooth}
Suppose that $X$ and $Y$ are algebraic stacks and $X \rightarrow Y$ is a smooth morphism of Deligne--Mumford type.  Then $\sN_{X/Y} \simeq B \sT_{X/Y}$.
\end{proposition}
\begin{proof}
It is sufficient to assume that $Y$ is affine (or at least of Deligne--Mumford type) and therefore that $X$ is a Deligne--Mumford stack.  We can construct the construct isomorphism on the small \'etale site of $X$.  Note in that case that a section of $\sN_{X/Y}(A)$, when $A$ is an \'etale $X$-algebra, is a square-zero extension $A'$ of $A$ as a $Y$-algebra.  The lifts of the $Y$-algebra structure on $A'$ to an $X$-algebra structure form a torsor on $\Spec A$ under the quasi-coherent sheaf of groups $T_{X/Y} \tensor_X \tJ$ where $J$ is the ideal of $A$ in $A'$.  But $\Gamma(A, T_{X/Y} \tensor_X \tJ) = \sT_{X/Y}(A, J)$, so this gives the desired map $\sN_{X/Y} \rightarrow B \sT_{X/Y}$.  It is well-known that this map induces a bijection on morphisms.  Since both $\sN_{X/Y}(A,J)$ and $B \sT_{X/Y}(A,J)$ are gerbes, this implies that the map is an equivalence.
\end{proof}

\subsection{Local complete intersection morphisms}

Let $X$ and $Y$ be algebraic stacks and let $X \rightarrow Y$ be a morphism of Deligne--Mumford type.  Then $\sN_{X/Y}$ is a relative obstruction theory for $X$ over $Y$, as we saw in Section~\ref{sec:ins}.

\begin{proposition}
If $X \rightarrow Y$ is a local complete intersection morphism, then $\sN_{X/Y}$ is perfect.
\end{proposition}
\begin{proof}[First proof.]
Remark that $\sN_{X/Y}(A, J) = \Ext(\bL_{X/Y} \tensor_X A, J)$ and $\bL_{X/Y}$ is perfect in degrees $[-1,0]$.  
\end{proof}
\begin{proof}[Second proof.]
First note that $X$ is locally of finite presentation over $Y$ and $\sN_{X/Y}$ is locally of finite presentation as an obstruction theory.  We can therefore use the criterion of Proposition~\ref{prop:perf}:  it suffices to see that the abelian cone stack $\fN_{X/Y}$ is a vector bundle stack.

It is sufficient to assume that $Y$ is an affine scheme by base change, and the problem is \'etale local in $X$, so we can also assume that $X$ is affine.  We can therefore assume that $X \rightarrow Y$ factors as a closed complete intersection embedding $i : X \rightarrow Z$ followed by a smooth map $Z \rightarrow Y$.  We have an exact sequence
\begin{equation*}
0 \rightarrow \fN_{X/Z} \rightarrow \fN_{X/Y} \rightarrow i^\ast \fN_{Y/Z}
\end{equation*}
and $\fN_{X/Z}$ is a vector bundle and $i^\ast \fN_{Y/Z} = i^\ast B T_{Y/Z}$ (Proposition~\ref{prop:smooth}) is a vector bundle stack .  Therefore by \cite[Compl\'ement~I.4.11]{perf}, $\fN_{X/Y}$ is a vector bundle stack.
\end{proof}

\subsection{Stable maps} \label{sec:stable-maps}

Let $X \rightarrow V$ be a \emph{smooth}, Deligne--Mumford type morphism of algebraic stacks.  Consider the stack $\fM(X/V)$ whose $S$-points are commutative diagrams
\begin{equation} \label{eqn:9} \xymatrix{
C \ar[r] \ar[d]_\pi & X \ar[d]^q \\
S \ar[r] & V
} \end{equation}
where $C / S$ is a family of twisted Deligne--Mumford pre-stable curves.  Let $\fM$ be the stack of twisted Deligne--Mumford pre-stable curves.  There is a projection $h : \fM(X/V) \rightarrow \fM \times V$.  We describe a relative obstruction theory.  For the description of the obstruction theory itself, there is no requirement that $C$ be a pre-stable curve; that hypothesis is only necessary to demonstrate that the obstruction theory is perfect.

We will pass to the opposite category here and define the relative obstruction theory as a fibered category over the category of affine schemes over $\fM(X/V)$.  Begin by defining $\sT_{X/V}^\circ$ as in Section~\ref{sec:tang} (but passing to the opposite category as we work geometrically in this section).  If $C$ is an $X$-scheme and $J$ is a quasi-coherent sheaf on $C$ then an object of $\sT_{X/V}^\circ(C,J)$ is a completion of the diagram
\begin{equation*} \xymatrix{
C \ar[r] \ar[d] & X \ar[d] \\
C[J] \ar[r] \ar@{-->}[ur] & V
} \end{equation*}
where $C[J]$ is the trivial square-zero extension of $C$ by $J$ and the map $C[J] \rightarrow V$ is the zero tangent vector (the map factoring through the retraction of $C[J]$ onto $C$).  


Recall that $B \sT^\circ_{X/V} = \sN^\circ_{X/V}$.  Given $S$-point~\eqref{eqn:9} of $\fM(X/V)$ and a quasi-coherent sheaf $J$ on $S$, we obtain a stack $\usN^\circ_{X/V}(C,\pi^\ast J)$ by restricting $\sN^\circ_{X/V}$ to $\et(C)$.  Define $\usE^\circ(S,J) = \pi_\ast \usN^\circ_{X/V}(C, \pi^\ast J)$.

To show that $\usE^\circ(S,J)$ is the (opposite of the) restriction of an obstruction theory $\sE^\circ$ to $S$ and $J$, we must describe how its functoriality with respect to the variation of $S$ and $J$.  It must vary both covariantly and contravariantly with $S$ and covariantly with $J$ (the contravariance in $S$ and covariance in $J$ combine to $\sE$'s being a cofibered category over $\ComRngMod$ and the contravariance with $S$ corresponds to being fibered over $\ComRng$).

For the contravariance in $S$ and the covariance in $J$, consider a commutative diagram
\begin{equation} \label{eqn:28} \xymatrix{
C_S \ar@/^12pt/[rr]^{\varphi_S} \ar[r]_g \ar[d] & C_T \ar[d] \ar[r]_{\varphi_T} & X \ar[d] \\
S \ar[r]_f & T \ar[r] & V
} \end{equation}
corresponding to a morphism in $\fM(X/V)$, and a morphism of $\cO_S$-modules $f^\ast I \rightarrow J$.  We assume that $S \rightarrow T$ is affine.  Since $\sN^\circ_{X/V}$ is an obstruction theory there is a map $\sN^\circ_{X/V}(C_T, \pi_T^\ast I) \rightarrow \sN^\circ_{X/V}(C_S, \pi_S^\ast J)$.  Allowing $C_S$ and $C_T$ to vary in their \'etale sites gives a map $\usN^\circ_{X/V}(C_T, \pi_T^\ast I) \rightarrow g_\ast \usN^\circ_{X/V}(C_S, \pi_S^\ast J)$.  Then pushing forward by $\pi_T$, we get $\usE^\circ(T, I) \rightarrow f_\ast \usE^\circ(S,T)$ and passing to global sections gives the map $\sE^\circ(T,I) \rightarrow \sE^\circ(S,J)$.

For the covariance in $S$, consider again the diagram~\eqref{eqn:28}.  We only need to consider the case where $S$ and $T$ are both affine (by definition), but it is actually enough to assume only that $f$ is an affine map.  Therefore $g$ is also affine.  We have given a $V$-extension of $\varphi_S^\ast \cO_X$ by $\pi_S^\ast J$, we may push forward by $g$ to obtain a $V$-extension of $g_\ast \varphi_S^\ast \cO_X$ by $g_\ast \pi_S^\ast J = \pi_T^\ast f_\ast J$.  Pulling back via the map $\varphi_T^\ast \cO_X \rightarrow g_\ast \varphi_S^\ast \cO_X = g_\ast g^\ast \varphi_T^\ast \cO_X$, we get an extension of $\varphi_T^\ast \cO_X$ by $\pi_T^\ast f_\ast J$, which is a section of $\sE^\circ(T, f_\ast J)$, as needed.

\begin{remark}
In fact, one can demonstrate more generally that if $\sE^\circ$ is an obstruction theory then its natural extension to all schemes (Section~\ref{sec:global}) behaves covariantly with respect to affine morphisms.
\end{remark}

Finally, we have to construct the cartesian diagram~\eqref{eqn:6} (or really, given that we are working now with the opposite category, a cartesian diagram
\begin{equation*} \xymatrix{
\Def_{\fM(X/V)} \ar[r] \ar[d] & e \ar[d] \\
h^\ast \Def_{\fM \times V} \ar[r]^<>(0.5){\omega} & \sE^\circ ,
} \end{equation*}
using the notation of Section~\ref{sec:global}).  An object of $h^\ast \Def_{\fM \times V}$ corresponds to a commutative diagram of solid lines
\begin{equation} \label{eqn:29} \xymatrix{
C \ar[r] \ar[d] \ar@/^15pt/[rr]^{\varphi} & C' \ar[d] \ar@{-->}[r] & X \ar[d]^q \\
S \ar[r] & S' \ar[r] & V
} \end{equation}
in which $S'$ is a square-zero extension of $S$, and to lift it to $\Def_{\fM(X/V)}$ means to construct a dashed arrow making the diagram commutative.  This immediately translates to the lifting problem
\begin{equation} \label{eqn:30} \xymatrix{
C \ar[r] \ar[d] & X \ar[d] \\
C' \ar@{-->}[ur] \ar[r] & V ,
} \end{equation}
which we know is obstructed by the obstruction theory $\sN^\circ_{X/V}$.  This gives the map $\Def_{\fM \times V} \rightarrow \sN^\circ_{X/V}$ and it is immediate from the fact that $\sN^\circ_{X/V}$ is an obstruction theory for $X$ over $V$ that $\sE^\circ$ is an obstruction theory for $\fM(X/V)$ over $\fM \times V$.

This can be put somewhat more precisely as follows.  We obtain the map $h^\ast \uDef_{\fM \times V}(S,J) \rightarrow \usE^\circ(S,J) = \pi_\ast \usN^\circ_{X/V}(C, \pi^\ast J)$ as the composition
\begin{equation*}
h^\ast \uDef_{\fM \times V}(S,J) \rightarrow \pi_\ast q^\ast \uDef_V(C,\pi^\ast J) \rightarrow \pi_\ast \usN^\circ_{X/V} .
\end{equation*}
The first arrow sends a map $S' \rightarrow \fM \times V$ extending the one induced from $S \rightarrow \fM(X/V)$ to the diagram~\eqref{eqn:30}.  Since $\sN^\circ_{X/V}$ is an obstruction theory for $X$ over $V$, we have
\begin{equation*}
\Def_X(C, \pi^\ast J) = q^\ast \Def_V(C, \pi^\ast J) \fp_{\usN^\circ_{X/V}(C, \pi^\ast J)} e .
\end{equation*}
Pushforward is exact, so we get
\begin{equation*}
\uDef_{\fM \times V}(S,J) \fp_{\usE^\circ(S,J)} e = \uDef_{\fM \times V}(S,J) \fp_{\pi_\ast q^\ast \uDef_V(C, \pi^\ast J)} \uDef_X(C, \pi^\ast J) .
\end{equation*}
Unwound, the last term is exactly the collection of completions of Diagram~\eqref{eqn:29}.  This is exactly what we needed to show that $\sE$ is an obstruction theory.

We can also demonstrate that the obstruction theory is perfect using the criterion for perfection proved above, provided $X$ is locally of finite presentation over $V$.  In effect, it is sufficient to remark that in this case $\sE$ is locally of finite presentation, and therefore it is enough to prove that $\pi_\ast \varphi^\ast \fN_{X/V}$ ($ = \pi_\ast \usN^\circ_{X/V}(C, \cO_{C})$) is a vector bundle stack.  But now we use the facts that $C$ has relative dimension $1$ over $C$, that $\usN^\circ_{X/V}(C,\cO_C) = B \usT^\circ_{X/V}(C, \cO_C)$ is representable by $\varphi^\ast T_{X/V}[1]$, and that $T_{X/V}$ is flat over $S$ to show that $R \pi_\ast \varphi^\ast T_{X/V}[1]$ is perfect in degrees $[-1,0]$.

\appendix

\section{Biextensions}

Let $\sC$ be the abelian category of sheaves of $\cO$-modules on a site.  For any $C \in \sC$, we can define two functors $\sC \times \sC \rightarrow \Sets$:
\begin{gather*}
(A,B) \mapsto \Hom(A, \uHom(B,C)) \\
(A,B) \mapsto \Hom(A \tensor B, C) .
\end{gather*}
These functors are of course isomorphic.  If $\fF$, and $\fG$ are left exact additively fibered categories on $\sC$, define $\Hom(\fE, \uHom(\fF, \fG))$ to be the category of completions of the diagram
\begin{equation*} \xymatrix{
\fE \times \fF \ar@{-->}[r] \ar[d] & \fG \ar[d] \\
\sC \times \sC \ar[r]^{\tensor} & \sC .
} \end{equation*}
We treat $\fE \tensor \fF$ as a placeholder so that we may write $\Hom(\fE \tensor \fF, \fG)$ to mean $\Gamma(\fE, \uHom(\fF, \fG))$.  There is an equivalence $\fE \tensor \fF \simeq \fF \tensor \fE$, meaning that we have
\begin{equation*}
\Hom(\fE \tensor \fF, \fG) \simeq \Hom(\fF \tensor \fE, \fG) .
\end{equation*}

Allowing $\fE$ to vary among representable additively fibered categories over $\sC$, we obtain a fibered category $\uHom(\fF,\fG)$ with fiber $\Hom(X, \uHom(\fF, \fG))$ over $X \in \sC$.

\begin{lemma}
The fibered category $\uHom(\fF,\fG)$ over $\sC$ is left exact and additively fibered.
\end{lemma}
\begin{proof}
First we check additivity.  We have
\begin{multline*}
\Hom(0, \uHom(\fF, \fG)) = \varprojlim_{Z \in \fF} \fG(0 \tensor Z)  = 0 \\
\shoveleft \Hom(X \oplus Y, \uHom(\fF, \fG)) = \varprojlim_{Z \in \fF} \fG( (X \oplus Y) \tensor Z )  
 \\ = \varprojlim_{Z \in \fF} \fG(X \tensor Z) \times \varprojlim_{Z \in \fF} \fG(Y \tensor Z) ) \\ = \Hom(X, \uHom(\fF, \fG)) \times \Hom(Y, \uHom(\fF, \fG)) .
\end{multline*}
To demonstrate left exactness, consider an exact sequence
\begin{equation*}
0 \rightarrow W \rightarrow X \rightarrow Y \rightarrow 0
\end{equation*}
in $\sC$.  Note that flat $\cO$-modules generate $\sC$, so we have
\begin{multline*}
\Hom(Y, \uHom(\fF, \fG)) = \varprojlim_{\substack{Z \in \fF \\ \text{flat}}} \fG(Y \tensor Z)  = \varprojlim_{\substack{Z \in \fF \\ \text{flat}}} \fG( [W \rightarrow X ]  \tensor Z) \\
= \uHom( [W \rightarrow X], \uHom(\fF, \fG) ).
\end{multline*}
\end{proof}

Let $B \cO$ be the left exact, additively cofibered category of extensions by $\cO$.  Then if $\fF$ is an additively fibered category over $\sC$, we define $\fF^\vee = \uHom(\fF, B \cO)$.  We have
\begin{equation*}
\Hom(\fE, \fF^\vee) \simeq \Hom(\fE \tensor \fF, B \cO) 
\end{equation*}
and in particular
\begin{equation*}
\Hom(\fF^\vee, \fF^\vee) \simeq \Hom(\fF \tensor \fF^\vee, B \cO) \simeq \Hom(\fF, \uHom(\fF^\vee, B \cO)) .
\end{equation*}
We also have a natural map $\fF \rightarrow (\fF^{\vee})^{\vee}$ by following $\id_{\fF^\vee}$ through the chain of equivalences above.

\begin{lemma} \label{lem:vbcx}
Suppose that $\fF$ is represented by a $2$-term complex of vector bundles $\bF_\bullet$.  Then $\fF^\vee$ is represented by the complex $\bF_\bullet^\vee[1]$.
\end{lemma}
\begin{proof}
$\Hom(\fF, B \cO)$ is the category of diagrams
\begin{equation} \label{eqn:10} \xymatrix{
& & & \bF_1 \ar[d] \ar[dl] \\
0 \ar[r] & \cO \ar[r] & X \ar[r] & \bF_0 \ar[r] & 0
} \end{equation}
where the second row is exact.  On the other hand, a section of $[\bF_1^\vee / \bF_0^\vee]$ can be viewed as an $\bF_1^\vee$-torsor $P$ and a $\bF_1^\vee$-equivariant map $P \rightarrow \bF_0^\vee$.  There is a canonical way to extend $P$ to an extension $W$ of $\cO$ by $\bF_1^\vee$ such that $P$ is the fiber of $W$ over the section $1 \in \cO$; the map $P \rightarrow \bF_0^\vee$ extends to a map $W \rightarrow \bF_1^\vee$ and we have a commutative diagram
\begin{equation} \label{eqn:21} \xymatrix{
0 \ar[r] & \bF_0^\vee \ar[r] \ar[d] & W \ar[r] \ar[dl] & \cO \ar[r] & 0 \\
& \bF_1^\vee 
} \end{equation}
in which the first row is exact.  Since $\bF_\bullet$ is a complex of locally free $\cO$-modules and $\cO$ is its own dual, diagrams of the forms~\eqref{eqn:10} and~\eqref{eqn:21} are exchanged by duality.
\end{proof}

\begin{corollary} \label{cor:vbdual}
If $\fF$ is a vector bundle stack then $\fF^\vee$ is also a vector bundle stack.
\end{corollary}

\begin{corollary} \label{cor:vbdd}
If $\fF$ is a vector bundle stack then the map $\fF \rightarrow (\fF^\vee)^\vee$ is an equivalence.
\end{corollary}

\bibliographystyle{halpha}
\bibliography{obs}

\end{document}